\documentclass[11pt]{amsart}
\usepackage[utf8]{inputenc}
\usepackage[english]{babel}
\usepackage[pdftex]{graphicx}
\usepackage{amssymb}
\usepackage{amsmath}
\usepackage{mathtools}
\usepackage{amsfonts}
\usepackage{amsthm}
\usepackage{amssymb}
\usepackage{bm}
\usepackage{mathrsfs}
\usepackage{appendix}
\usepackage{enumerate}
\usepackage[left=3.25cm, right=3.25cm, top=3.5cm, bottom=3.5cm]{geometry}
\usepackage[stretch=30,shrink=30]{microtype}
\usepackage[usenames, dvipsnames]{color}
\allowdisplaybreaks[1]

\usepackage[doi=false,isbn=false,url=false,eprint=false,backref=true,backend=bibtex,giveninits=true,sorting=nyt]{biblatex}
\DeclareFieldFormat[article,book]{title}{\textbf{#1}}
\DeclareFieldFormat[book,inbook]{title}{\textbf{#1}}
\DeclareFieldFormat[inbook]{chapter}{#1}
\renewbibmacro{in:}{}
\DeclareFieldFormat[article]{volume}{vol\adddot\addnbspace #1}
\DeclareFieldFormat[article]{number}{no\adddot\addnbspace #1}
\renewbibmacro{volume+number+eid}{%
	\setunit{\addcomma\space}%
	\printfield{volume}%
	\setunit{\addcomma\space}%
	\printfield{number}%
	\setunit{\addcomma\space}%
	\printfield{eid}%
}
\usepackage[colorlinks=true, linkcolor=black, citecolor=black]{hyperref}
\usepackage[capitalise,nameinlink]{cleveref}
\crefformat{equation}{#2(#1)#3}
\crefrangeformat{equation}{#3(#1)#4 to~#5(#2)#6}
\crefmultiformat{equation}{#2(#1)#3}{ and~#2(#1)#3}{, #2(#1)#3}{ and~#2(#1)#3}

\makeatletter
\def\th@plain{%
	\thm@notefont{}
	\itshape 
}
\def\th@definition{%
	\thm@notefont{}
	\normalfont 
}
\makeatother

\numberwithin{equation}{section}

\newtheorem{theorem}{Theorem}[section]

\newtheorem{proposition}[theorem]{Proposition}
\newtheorem{corollary}[theorem]{Corollary}

\newtheorem{lemma}[theorem]{Lemma}

\theoremstyle{definition}
\newtheorem{definition}[theorem]{Definition}
\theoremstyle{remark}
\newtheorem{remark}[theorem]{Remark}

\newcommand{\de}{\partial}
\newcommand{\N}{\mathbb{N}}

\newcommand{\R}{\mathbb{R}}
\newcommand{\C}{\mathbb{C}}

\newcommand{\mres}{\mathbin{\vrule height 1.6ex depth 0pt width
		0.13ex\vrule height 0.13ex depth 0pt width 1.3ex}}
\newcommand{\ang}[1]{\langle #1 \rangle}
\newcommand{\mz}{\frac{1}{2}}
\newcommand{\inj}{\operatorname{inj}(M)}

\newcommand{\nin}{\not\in}

\renewcommand{\epsilon}{\varepsilon}
\newcommand{\dvol}{d\mathrm{vol}}

\DeclareMathOperator{\spt}{spt}
\DeclareMathOperator{\dist}{dist}

\DeclareMathOperator{\divergence}{div}
\DeclareMathOperator{\Sym}{Sym} 
\DeclareMathOperator{\tr}{tr}
\DeclareMathOperator{\Ric}{Ric}

\DeclareMathOperator{\End}{End}

\title[Parabolic $U(1)$-Higgs equations and mean curvature flows]{The parabolic $\bm{U(1)}$-Higgs equations and codimension-two mean curvature flows}
\author[D. Parise]{Davide Parise}
\address{University of California San Diego, Department of Mathematics, 9500 Gilman Drive \# 0112, La Jolla, CA 92093, United States of America}
\email{dparise@ucsd.edu}
\author[A. Pigati]{Alessandro Pigati}
\address{Bocconi University, Department of Decision Sciences, Via Guglielmo R\"ontgen 1, 20136 Milano, Italy}
\email{alessandro.pigati@unibocconi.it}
\author[D. Stern]{Daniel Stern}
\address{Cornell University, Department of Mathematics, 310 Malott Hall, Ithaca, NY 14853, United States of America}
\email{daniel.stern@cornell.edu}

\begin{document}

\begin{abstract}
We develop the asymptotic analysis as $\epsilon\to 0$ for the natural gradient flow of the self-dual $U(1)$-Higgs energies 
$$E_{\epsilon}(u,\nabla)=\int_M\left(|\nabla u|^2+\epsilon^2|F_{\nabla}|^2+\frac{(1-|u|^2)^2}{4\epsilon^2}\right)$$
on Hermitian line bundles over closed manifolds $(M^n,g)$ of dimension $n\ge 3$, showing that solutions converge in a measure-theoretic sense to codimension-two mean curvature flows---i.e., integral $(n-2)$-Brakke flows---generalizing results of \cite{PigatiStern} from the stationary case. Given any integral $(n-2)$-cycle $\Gamma_0$ in $M$, these results can be used together with the convergence theory developed in \cite{gammaConv} to produce nontrivial integral Brakke flows starting at $\Gamma_0$ with additional structure, similar to those produced via Ilmanen's elliptic regularization.
\end{abstract}
\maketitle

\tableofcontents 

\section{Introduction} 

Families of submanifolds $\Sigma_t\subset M$ moving by mean curvature inside an ambient Riemannian manifold $(M^n,g)$ tend to develop singularities in finite time, and one of the main challenges in the study of mean curvature flow 
is finding natural ways to continue the flow through singularities. The first and most general notion of a weak solution for mean curvature flow was introduced by Brakke \cite{Brakke}, who identified a natural extension to the setting of varifolds, satisfying desirable compactness and partial regularity properties. In general these Brakke flows are highly non-unique, allowing for pathological behaviors like instantaneous vanishing of the flow, but within this large class of weak solutions one can find distinguished flows with more regular behavior via natural approximation schemes, like Ilmanen's elliptic regularization \cite{EllipticReg}.

For hypersurfaces, another natural regularization of the mean curvature flow comes from the parabolic Allen--Cahn equation
\begin{equation}\label{pac}
	\frac{\partial u^{\epsilon}}{\partial t}=\Delta u^{\epsilon}-\frac{W'(u^{\epsilon})}{\epsilon^2}
\end{equation}
for scalar functions $u^{\epsilon}:M\times[0,\infty)\to \mathbb{R}$, where $W: \mathbb{R}\to [0,\infty)$ is a double-well potential like $W(u)=\frac{1}{4}(1-u^2)^2.$
First studied as a model for phase separation \cite{AllenCahn}, \cref{pac} arises variationally as the gradient flow for the energy functional $E_{\epsilon}(u)=\int_M(\frac{\epsilon}{2}|du|^2+\frac{W(u)}{\epsilon})$, which has long been known to approximate the area functional for hypersurfaces in a weak sense as $\epsilon\to 0$ \cite{ModicaMortola, Sternberg}. In the 1990s, a series of papers \cite{EvansSonerSouganidis, BronsardKohn, Ilmanen} confirmed a long-suspected link between \cref{pac} and mean curvature flow of hypersurfaces, culminating in Ilmanen's proof that the energy measures $(\frac{\epsilon}{2}|du^{\epsilon}|^2+\frac{W(u^{\epsilon})}{\epsilon})\,\dvol_g$ for solutions of \cref{pac} converge to codimension-one rectifiable Brakke flows as $\epsilon\to 0$ \cite{Ilmanen}. Subsequent work of Tonegawa proved the integrality of these limiting Brakke flows (up to a universal constant depending on $W$) \cite{Tonegawa}, and more recently Hensel--Laux established a weak-strong uniqueness property, showing that Brakke flows obtained from the Allen--Cahn regularization coincide with classical mean curvature flow wherever the latter is defined \cite{HenselLaux}. Meanwhile, the stationary case of these results, linking critical points of $E_{\epsilon}$ to minimal hypersurfaces as in \cite{HutchinsonTonegawa, TonWick}, has recently seen a number of exciting applications to the existence theory for minimal hypersurfaces in general Riemannian manifolds (see, e.g., \cite{Guaraco, GasparGuaraco, ChoMan1, ChoMan2}, among others). 

In the present paper, building on the work of \cite{PigatiStern, gammaConv}, we introduce an analog of the Allen--Cahn regularization for \emph{codimension-two} mean curvature flows via a natural parabolic system arising in gauge theory. More precisely, on a Hermitian line bundle $L\to M$ over a Riemannian manifold $M^n$, we consider families $(u_t,\nabla_t)=(u_t^\epsilon,\nabla_t^\epsilon)$ of sections $u_t\in \Gamma(L)$ and metric-compatible connections $\nabla_t=\nabla_0-i\alpha_t$ evolving by the nonlinear parabolic system 
\begin{align} \label{gf}
	\left\{
	\begin{aligned} 
		\partial_t u_t&=-\nabla_t^*\nabla_tu_t+{\textstyle\frac{1}{2\epsilon^2}}(1-|u_t|^2)u_t, \\
		\partial_t\alpha_t&=-d^*\omega_t+\epsilon^{-2}\langle iu_t,\nabla_tu_t\rangle,
	\end{aligned}
	\right.
\end{align}
giving the $L^2$-gradient flow for the \emph{self-dual $U(1)$-Higgs functionals}
$$E_{\epsilon}(u,\nabla)=\int_M \left( |\nabla u|^2+\epsilon^2|F_{\nabla}|^2+ \frac{(1-|u|^2)^2}{4\epsilon^2} \right) \, \dvol_g;$$
here we are following the convention of \cite{PigatiStern}, identifying the curvature $F_{\nabla_t}$ of the metric-compatible connection $\nabla_t$ with a real-valued two-form $\omega_t$ via $F_{\nabla_t}=-i\omega_t$. Originating in the study of superconductivity, the functionals $E_{\epsilon}$ and their critical points have received considerable attention from the gauge theory community since the 1980s, particularly when $M$ is a surface or a K\"ahler manifold \cite{Taubes1, Taubes2, JaffeTaubes, GarciaPrada, Bradlow}. In \cite{PigatiStern}, the second- and third-named authors studied the asymptotic behavior as $\epsilon\to 0$ of critical points for $E_{\epsilon}$ on arbitrary higher-dimensional manifolds in the natural bounded-energy regime $E_{\epsilon}(u_{\epsilon},\nabla_{\epsilon})=O(1)$, proving that solutions concentrate along stationary integral $(n-2)$-varifolds, and applying this to obtain a new construction of codimension-two minimal varieties in closed Riemannian manifolds. The main theorem of the present paper extends the asymptotic analysis of \cite{PigatiStern} from the stationary setting to the parabolic system \cref{gf}, providing a codimension-two analog for the results of Ilmanen and Tonegawa \cite{Ilmanen, Tonegawa} in the Allen-Cahn setting.

\begin{theorem} \label{thmBrakke}
	Let $(M^n,g)$ be a closed, oriented, Riemannian manifold of dimension $n \geq 3$, and let $L\to M$ be a Hermitian line bundle over $M$. Let $(u^{\epsilon}_t, \nabla_t^{\epsilon}=\nabla_0^{\epsilon}-i\alpha^{\epsilon}_t)_{t \geq 0}$ solve \cref{gf} with smooth initial condition $(u^{\epsilon}_0, \nabla^{\epsilon}_0),$ with $E_\epsilon(u^{\epsilon}_0, \nabla^{\epsilon}_0) \leq \Lambda$. 
	Then, there exist a subsequence $(\epsilon_i)_{i \in \mathbb{N}}$ with $\epsilon_i\to 0$ and a family of Radon measures $(\mu_t)_{t \geq 0}$ on $M$ such that 
	\begin{enumerate}[(i)]
		\item we have the weak-$*$ convergence of Radon measures 
		$$\mu_t^{\epsilon}:=\left(|\nabla_t^{\epsilon}u_t^{\epsilon}|^2+\epsilon^2|F_{\nabla_t^{\epsilon}}|^2+\frac{(1-|u_t^{\epsilon}|^2)^2}{4\epsilon^2}\right)\,\dvol_g \rightharpoonup^* \mu_t$$ for all $t > 0$; 
		\item for $\mathcal{L}^1$-a.e.\ $t > 0$ the measures $\mu_t$ are integer $(n - 2)$-rectifiable; 
		\item the family $(\mu_t)_{t \geq 0}$ defines a Brakke flow, as described in \cref{brakke.def}. 
	\end{enumerate} 
\end{theorem}

\begin{remark}
	While we have opted to work on compact manifolds without boundary---in view of the applications we have in mind, and for simplicity of presentation---most of the analysis is essentially local, and with some additional work one should be able to obtain natural analogs of \cref{thmBrakke} on manifolds with boundary (with suitable boundary conditions) and a large class of complete noncompact manifolds (in particular $\mathbb{R}^n$) under natural constraints on the initial data.
\end{remark}

Our proof of rectifiability and Brakke's inequality follow a slightly different route compared to \cite{Ilmanen}, since we do not have an obvious diffuse analogue of the generalized mean curvature $\vec{H}$ for each pair $(u^\epsilon_t,\nabla_t^\epsilon)$. For instance, in order to show rectifiability, rather than bounding the mean curvature of $\mu_t$, we perform a blow-up analysis at carefully chosen (but in some sense ``generic'') points $(x,t)$ in the spacetime such that the dilated sequence becomes almost stationary. In the limit we obtain a (generalized) stationary varifold, which turns out to be rectifiable (via the analysis of \cite{AmbrosioSoner}), and in particular a constant multiple of an $(n-2)$-plane. The proof of integrality also differs slightly from \cite{Tonegawa}, in that it relies more directly on the study of suitable two-dimensional slices, where $(u^\epsilon_t,\nabla_t^\epsilon)$ (after dilation) resembles an entire critical point on the plane.

Prior to publication, \cref{thmBrakke} has already seen applications to the study of critical points for the functionals $E_{\epsilon}$. In particular, the asymptotic analysis furnished by \cref{thmBrakke} provides a crucial ingredient in the work of De Philippis and the second-named author in the construction of critical points $(u^{\epsilon},\nabla^{\epsilon})$ converging to prescribed non-degenerate minimal submanifolds of codimension two in \cite{GuidoPigati}, a result which provides a kind of converse to the asymptotic analysis of critical points obtained in \cite{PigatiStern}, analogous to results of Pacard--Ritor\'e in the Allen--Cahn setting \cite{PacRit} (see also \cite{BdP} for an alternative approach via gluing methods).

As in the stationary case \cite{PigatiStern}, it is instructive to compare the results of \cref{thmBrakke} with related work on the (un-gauged) parabolic Ginzburg--Landau equations
\begin{equation}\label{cgl}
	\frac{\partial v_t^{\epsilon}}{\partial t}=\Delta v_t^{\epsilon}-\epsilon^{-2}(1-|v_t^{\epsilon}|^2)v_t^{\epsilon}
\end{equation}
for complex-valued maps $v_t^{\epsilon}:M\to \mathbb{C}$, which arise as the gradient flow for the functionals
$$F_{\epsilon}(v):=\int_M\left(\frac{1}{2}|dv|^2+\frac{(1-|v|^2)^2}{4\epsilon^2}\right).$$
In \cite{BethuelOrlandiSmets}, building on the analysis of \cite{AmbrosioSoner, BethuelBrezisOrlandi, LinRiviere}, it is shown that for families of maps $v^{\epsilon}_t\in C^{\infty}(M,\mathbb{C})$ evolving via \cref{cgl} with the natural energy bound $F_{\epsilon}(v^{\epsilon}_t)=O(\lvert\log\epsilon\rvert)$, the energy measures $\frac{|dv^{\epsilon}_t|^2}{\pi\lvert\log\epsilon\rvert}\,\dvol_g$ converge as $\epsilon\to 0$ to a sum of a diffuse measure and an $(n-2)$-rectifiable (not necessarily integral) Brakke flow. While the diffuse component of the limiting measure vanishes under mild assumptions on the initial data \cite[Theorem~C]{BethuelOrlandiSmets}, the results of \cite{PigatiStern.quant} show that integrality of the limiting Brakke flow should not be expected in general (not even in the stationary case) due to long-range interactions between distant components of the ``vorticity set'' $\{|v^{\epsilon}_t|\leq \frac{1}{2}\}$. In particular, while the results of \cite{AmbrosioSoner, BethuelOrlandiSmets} represent a major achievement in the study of the parabolic Ginzburg--Landau equations, they should not quite be thought of as codimension-two analogs of \cite{Ilmanen, Tonegawa}, in light of the fundamental qualitative differences between the behavior of solutions for \cref{cgl} and their (formally identical) scalar counterparts. 

By contrast, the asymptotic analysis of the system \cref{gf} (as in the stationary case \cite{PigatiStern}) bears striking similarities to that of the parabolic Allen--Cahn equation in \cite{Ilmanen,Tonegawa}: energy concentrates in a (parabolic) $O(\epsilon)$-neighborhood of the zero set $\{u_t^{\epsilon}=0\}$ and decays exponentially away from that region, and a central role in the analysis is played by the ``discrepancy'' $\epsilon^2|F_{\nabla_t^{\epsilon}}|^2-\frac{(1-|u_t^\epsilon|^2)^2}{4\epsilon^2}$ between the Yang--Mills and potential components of the energy density, mirroring the role of the discrepancy between Dirichlet and potential components in the Allen--Cahn setting \cite{Ilmanen, Tonegawa}. In particular, as discussed in \cite{gammaConv}, controlling this discrepancy is the key to obtaining a monotonicity result modeled on that of Huisken in the mean curvature flow setting \cite{Huisk} for the system \cref{gf}.

One important feature of the Brakke flows arising from \cref{gf}, shared by those obtained from the system \cref{cgl} and those constructed by elliptic regularization \cite{EllipticReg}, is the existence of an additional current structure for the spacetime track, allowing one to rule out instantaneous vanishing of the Brakke flows obtained in \cref{thmBrakke} for natural choices of initial data. Namely, we have the following result, showing existence of \emph{enhanced motions} arising from \cref{gf} for an arbitrary initial $(n-2)$-cycle (cf.\ \cite[Theorem~D]{BethuelOrlandiSmets} for solutions of \cref{cgl}).

\begin{theorem} \label{thmEnhanced}
	Let $L\to M$ be as above, and let $\Gamma_0$ be an integral $(n - 2)$-cycle Poincaré dual to the Euler class $c_1(L) \in H^2(M; \mathbb{Z})$. Then there exist a sequence of solutions $(u^{\epsilon}_t,\nabla^{\epsilon}_t)$ of \cref{gf}
	and an integer multiplicity $(n - 1)$-current $\Gamma\in {\bf I}_{loc}^{n-1}( M\times[0,\infty))$ such that $\partial \Gamma = \Gamma_0$, 
	$$\mu_0 =\lim_{\epsilon\to 0}e^{\epsilon}(u^{\epsilon}_0,\nabla^{\epsilon}_0)\,\dvol_g= 2 \pi \vert \Gamma_0 \vert,$$
	and denoting by $\mu_t$ the associated limiting Brakke flow, the pair $(\Gamma, \frac{1}{2\pi} \mu_t)$ defines an enhanced motion in the sense of \cite{EllipticReg} (see \cref{defEnh} below). In particular, $\frac{1}{2\pi}\mu_t\geq |\Gamma_t|$ for almost every time-slice $\Gamma_t$ of $\Gamma$.
\end{theorem}

The proof is similar in spirit to that of \cite[Theorem~D]{BethuelOrlandiSmets}, with the convergence theory developed in \cite{gammaConv} playing a role analogous to that of \cite{JSgamma} in \cite{BethuelOrlandiSmets}. More precisely, for any initial integral $(n-2)$-cycle $\Gamma_0$, we find a sequence $(u_0^{\epsilon},\nabla_0^{\epsilon})$ of initial conditions for \cref{gf} with energy concentrating along $\Gamma_0$ and for which the gauge-invariant Jacobian two-forms $\frac{1}{2\pi}J(u_0^{\epsilon},\nabla_0^{\epsilon})$ defined in \cite{gammaConv} (and \cref{pre.sec} below) converge as $(n-2)$-currents to $\Gamma_0$. We then extend the gradient flow $(u^{\epsilon}_t,\nabla^{\epsilon}_t)$ with initial condition $(u^{\epsilon}_0,\nabla^{\epsilon}_0)$ to a family $(\tilde{u}^{\epsilon},\widetilde{\nabla}^{\epsilon})$ on the pullback bundle of $L$ over $ M\times[0,\infty)$, and obtain the $(n-1)$-current $\Gamma$ as a distributional limit of $\frac{1}{2\pi}J(\tilde{u}^{\epsilon},\widetilde{\nabla}^{\epsilon})$, again appealing to the convergence theory of \cite{gammaConv}.

In particular, since the currents $\Gamma_t$ cannot vanish instantaneously, it follows that the $U(1)$-Higgs regularization (like Ilmanen's elliptic regularization \cite{EllipticReg}) yields nontrivial integral Brakke flows starting from any initial $(n-2)$-cycle $\Gamma_0$. Note that if $c_1(L)\neq 0$, it follows that the associated Brakke flow never vanishes, since the time-slices $\Gamma_t$ all lie in the same nontrivial homology class.

\begin{remark}
	Of course, one can hope for further improvements on the structure of the flows produced in \cref{thmEnhanced}: for instance, one expects that a weak-strong uniqueness result analogous to that of \cite{HenselLaux} in the Allen--Cahn setting should hold for the codimension-two enhanced flows obtained from \cref{thmEnhanced}; relatedly, it seems likely that a Brakke-type regularity result analogous to that obtained in \cite{NguyenWang} holds before passing to the limit $\epsilon\to 0$ in this setting as well.
\end{remark}

\begin{remark}
	In the case where $M$ is a K\"ahler-Einstein manifold of complex dimension two (real dimension four) and $\Gamma_0$ is an immersed \emph{Lagrangian} surface, a well-known result of Smoczyk \cite{Smoczyk} shows that the Lagrangian condition is preserved by the smooth mean curvature flow before singularities form. It would be very interesting to understand whether any natural approximation of the Lagrangian condition is preserved by the $U(1)$-Higgs gradient flow \cref{gf}, at least in the settings relevant to the Thomas--Yau conjecture \cite{ThomasYau}, providing a gauge-theoretic approximation of Lagrangian mean curvature flow for surfaces. In this direction, observe that there is a natural $U(1)$-Higgs counterpart to special Lagrangian surfaces in the hyperk\"ahler setting, corresponding to solutions of the vortex equations considered in \cite{Bradlow,GarciaPrada} after a suitable hyperk\"ahler rotation.
\end{remark}

\subsection{Outline of the paper}
In \cref{pre.sec}, we record some basic properties of solutions to \cref{gf}, and review the definitions of Brakke flow and related concepts needed for a precise formulation of Theorems \ref{thmBrakke}, \ref{thmEnhanced}, and their proofs. In \cref{secEpsIneq}, we show how to obtain uniform bounds on the difference $\epsilon|F_{\nabla}|-\frac{(1-|u|^2)}{2\epsilon}$ for solutions of \cref{gf} in terms of initial energy bounds, and use these to obtain a sharp Huisken-type monotonicity result, improving upon some of the estimates obtained in \cite[Section~6]{gammaConv}. In \cref{decay.sec}, we establish crucial exponential decay and ``clearing-out" results for the energy density of solutions to \cref{gf}, showing that energy concentrates in a parabolic $O(\epsilon)$-neighborhood of the set $\{|u|<1-\beta\}$ for some universal constant $\beta$, and giving certain lower bounds on the $(n-2)$-density of the limiting energy measures $\mu_t$. By studying tangent flows to the limiting family $\mu_t$ at generic points in the support of $\mu$, and relating the small-scale behavior near these points to entire stationary solutions of \cref{gf} on $\mathbb{R}^2$, we establish in Sections \ref{tf.sec} and \ref{sec.int} the rectifiability and integrality, respectively, of the measures $\mu_t$. In \cref{brakke.sec}, we complete the proof of \cref{thmBrakke} by showing that Brakke's inequality is satisfied by the family $\mu_t$, and in \cref{EnMo} we prove \cref{thmEnhanced}, by combining the preceding analysis with the $\Gamma$-convergence theory developed in \cite{gammaConv}.

\subsection*{Acknowledgements}
The authors thank Salvatore Stuvard and Yoshihiro Tonegawa for answering some of their questions about Brakke flows. DS acknowledges the support of the National Science Foundation during the completion of this work through the fellowship DMS-2002055.

\section{Preliminaries and notation}\label{pre.sec} 
Let $L\to M$ be a Hermitian line bundle---equivalently, a rank-two real vector bundle with Euclidean structure and almost complex structure---over a closed, oriented Riemannian manifold $(M^n, g)$ of dimension $n\geq 3$. For $\epsilon>0$, consider the energies
$$E_{\epsilon}(u,\nabla)=\int_M \left( |\nabla u|^2+\epsilon^2|F_{\nabla}|^2+ \epsilon^{-2}W(u) \right) \, \dvol_g,$$
for couples $(u, \nabla)$ of sections $u\in \Gamma(L)$ and metric-compatible connections $\nabla$ on $L$. The nonlinear potential $W \colon L \rightarrow \mathbb{R}$ is given by 
\begin{equation}
	W(u) = \frac{1}{4}(1 - \vert u \vert^2)^2, 
\end{equation}
and $F_\nabla \in \Omega^2(\End(L))$ denotes the curvature of the connection $\nabla$. We identify the curvature $F_\nabla$ with a real, closed two-form $\omega = \omega_\nabla $ via 
\begin{equation*}
	F_\nabla(X, Y) = [\nabla_X, \nabla_Y]u - \nabla_{[X, Y]}u = - i \omega_\nabla(X, Y)u,  
\end{equation*}
for vector fields $X$ and $Y$ on $M$. Recall that the functional $E_\epsilon$ is gauge invariant in the following sense: denoting by $\mathcal{G}$ the group of smooth maps $M\to S^1$ (with the operation of pointwise multiplication), for any $\phi\in\mathcal{G}$ the energy $E_\epsilon$ is invariant under the change of gauge 
\begin{equation*}
	\phi \cdot (u, \nabla) = (\phi u, \nabla - i \phi^*(d\theta)), 
\end{equation*}
corresponding to a fiberwise rotation of $L$.

We will say that the smooth couples $(u_t,\nabla_t=\nabla_0-i\alpha_t)_{t\in[0,\infty)}$ solve the gradient flow equations for $E_{\epsilon}$ if they satisfy the coupled nonlinear heat equations
\begin{align*}
	\left\{
	\begin{aligned} 
		\partial_t u_t&=-\nabla_t^*\nabla_tu_t+{\textstyle\frac{1}{2\epsilon^2}}(1-|u_t|^2)u_t, \\
		\partial_t\alpha_t&=-d^*\omega_t+\epsilon^{-2}\langle iu_t,\nabla_tu_t\rangle,
	\end{aligned}
	\right.
\end{align*}
where $\nabla^*$ denotes the formal $L^2$-adjoint of $\nabla$ and $d^*$ the formal $L^2$-adjoint of $d$. Long-time existence, uniqueness and continuous dependence on the initial data for this flow have been established in \cite[Section~6]{gammaConv}. 

The system \cref{gf} is, formally, the gradient flow of $\frac{1}{2} E_\epsilon$ with respect to the $L^2$-inner product 
\begin{equation*}
	\langle (u, \alpha), (v, \beta) \rangle = \int_M (\langle u, v \rangle + \epsilon^2 \langle \alpha, \beta \rangle)\,\dvol_g,
\end{equation*}
and it is easy to see that solutions of \cref{gf} satisfy the energy identity 
\begin{equation}\label{enerID}
	E_\epsilon(u_t, \nabla_t) + 2 \int_s^t \int_M \left( \vert \dot{u}_\tau \vert^2 + \epsilon^2 \vert \dot{\alpha}_\tau \vert^2 \right) \, \dvol_g \, d\tau = E_\epsilon(u_s, \nabla_s), 
\end{equation}
where $s < t$. 
As an obvious consequence, energy is decreasing along the flow, so if the initial data $(u_0,\nabla_0)$ satisfies the energy bound $E_{\epsilon}(u_0,\nabla_0)\leq \Lambda, $ we have $E_{\epsilon}(u_t,\nabla_t)\leq \Lambda$ for all $t>0$. We sometimes write $(u_t^\epsilon,\nabla_t^\epsilon)$ to emphasize that the pair depends on $\epsilon$, which typically varies along a sequence $\epsilon_j\to 0$.

Throughout this work we will write
\begin{equation*}
	e_t^\epsilon=e^\epsilon(u_t, \nabla_t) := |\nabla_t u_t|^2+\epsilon^2|F_{\nabla_t}|^2+\frac{(1-|u_t|^2)^2}{4\epsilon^2}
\end{equation*}
for the energy density associated to the pair $(u_t,\nabla_t)$. For any fixed $t\geq 0$, we denote by 
\begin{equation*}
	d\mu_t^\epsilon := e^\epsilon(u_t^\epsilon, \nabla_t^\epsilon)\,\dvol_g,  
\end{equation*}
the associated energy measure on $M$, and denote the energy measure on $M\times [0,\infty)$ by
\begin{equation*}
	\mu^\epsilon := \mu_t^{\epsilon} \otimes \mathcal{L}^1(t) , \text{ i.e., }d\mu^\epsilon=d\mu_t^\epsilon \otimes dt,
\end{equation*}
where $\mathcal{L}^1$ denotes the one-dimensional Lebesgue measure. 
It is also convenient to introduce the notation
\begin{equation*}
	d\nu^\epsilon:=(|\dot u_t^\epsilon|^2+\epsilon^2|\dot\alpha_t^\epsilon|^2)\, \dvol_g \otimes dt
\end{equation*}
for the measure on $M\times [0,\infty)$ appearing on the right-hand side of \cref{enerID}, which quantifies the non-stationarity of the pair $(u_t,\nabla_t)$ in an $L^2$ sense.

As in \cite{PigatiStern} and \cite{gammaConv}, for a pair $(u, \nabla)$, let $\psi(u, \nabla)$ be the two-form defined by 
\begin{equation*}
	\psi(u, \nabla)(X, Y) := 2 \langle i \nabla_X u, \nabla_Y u \rangle, 
\end{equation*}
for vector fields $X$ and $Y$, and define the gauge-invariant Jacobian 
\begin{equation}\label{j.def}
	J(u, \nabla) := \psi(u, \nabla) + (1 - \vert u \vert^2) \omega_\nabla. 
\end{equation}
The forms $J(u,\nabla)$ play a central role in the $\Gamma$-convergence theory developed in \cite{gammaConv}, and in \cref{EnMo}, we will invoke those results to obtain the current structure for the spacetime track described in \cref{thmEnhanced}.

\subsection{B\"ochner formulae}
Next, we recall from \cite[Section~6]{gammaConv} some of the key parabolic B\"ochner--Weitzenb\"ock formulae for solutions of \cref{gf}, which we use repeatedly in the proof of \cref{thmBrakke}. Given a family of pairs $(u_t,\nabla_t)$ solving \cref{gf} with curvature two-forms $F_{\nabla_t}=-i\omega_t$, denoting $\Delta_H=d^*d+dd^*$ the (positive semidefinite) Hodge Laplacian, we have
\begin{equation}
	\epsilon^2(\partial_t+\Delta_H)\omega_t=\psi(u_t,\nabla_t)-|u_t|^2\omega_t,
\end{equation}
and taking the inner product with $\omega_t$ gives the parabolic Bochner identity
\begin{equation}\label{omega.boch}
	-\epsilon^2(\partial_t+d^*d)\frac{1}{2}|\omega_t|^2=|u_t|^2|\omega_t|^2+\epsilon^2|D\omega_t|^2-\langle \psi(u_t,\nabla_t),\omega_t\rangle+\epsilon^2\mathcal{R}_2(\omega_t,\omega_t),
\end{equation}
where $\mathcal{R}_2$ denotes the Weitzenb\"ock curvature operator for two-forms. Likewise, taking the inner product with $u_t$ in the first equation of \cref{gf} gives
\begin{equation}\label{u.boch}
	-(\partial_t+d^*d)\frac{1}{2}|u_t|^2=|\nabla_t u_t|^2-\frac{1}{2\epsilon^2}(1-|u_t|^2)|u_t|^2, 
\end{equation}
from which it follows that 
\begin{equation} \label{Wboch}
	-(\de_t+d^*d)\frac{1-|u_t|^2}{\epsilon}=\frac{|u_t|^2}{\epsilon^2}\left(\frac{1-|u_t|^2}{\epsilon} \right)-\frac{2}{\epsilon}|\nabla_t u_t|^2. 
\end{equation}
Notice that from \cref{u.boch} and maximum principle it follows that $\vert u_t \vert \leq 1$, provided $\vert u_0 \vert \leq 1$. For the Dirichlet term $|\nabla_tu_t|^2$, we have 
\begin{align} \label{NablaSquared.boch}\begin{aligned}
		&-(\de_t+dd^*)\frac{1}{2}|\nabla_t u_t|^2 \\
		&=|\nabla_t^2u_t|^2+\frac{3|u_t|^2-1}{2\epsilon^2}|\nabla_t u_t|^2-2\ang{\omega_t,\psi(u_t,\nabla_t)}+\mathcal{R}_1(\nabla_t u_t,\nabla_t u_t), 
\end{aligned}\end{align}
where at $p \in M$ we have $\mathcal{R}_1(\nabla u, \nabla u) = \Ric(e_i, e_j) \langle \nabla_{e_i} u, \nabla_{e_j} u \rangle $.

\subsection{Brakke flows and enhanced motions}\label{vari.prelim}
Next, we review the concepts of Brakke flow and generalized Brakke flow as we will use them; here and elsewhere in the paper, we assume some familiarity with the theory of varifolds, as presented in, e.g., \cite{SimonGMT}.

In the proof of \cref{thmBrakke}, we will be adopting Ilmanen's definition of Brakke flow, which differs slightly from Brakke's original one (see \cite[Section~6]{EllipticReg}). Let $\phi \in C^2_c(M, \mathbb{R}_{\geq 0})$, and let $\mu$ be the weight measure associated to a $k$-rectifiable varifold $V_\mu$, with first variation $\delta V_\mu$ and total variation measure $\vert \delta V_\mu \vert$. If $\vert \delta V_\mu \vert$ is a Radon measure, we denote by $\vert \delta V_\mu \vert_{\mathrm{sing}}$ its singular part with respect to $\mu$, i.e., the restriction $\vert \delta V_\mu \vert \mres \{d \vert \delta V_\mu \vert/d\mu = \infty\}$, and let $-\vec{H}_\mu:=d(\delta V_\mu)/d\mu$ be the density of $\delta V_\mu$ with respect to $\mu$. Whenever $\mu$ and $\phi$ satisfy the four conditions
\begin{enumerate}[(i)]
	\item $\mu \mres \{\phi > 0 \}$ is a $k$-rectifiable Radon measure; 
	\item $\vert \delta V_\mu \vert \mres \{\phi > 0\}$ is a Radon measure;
	\item $\vert \delta V_\mu \vert_{\mathrm{sing}} \mres \{\phi > 0\} = 0$; 
	\item $\int \phi |\vec H_\mu|^2 < \infty$; 
\end{enumerate}
we define the operator 
\begin{equation*}
	\mathcal{B}(\mu, \phi) := \int_M (- \phi(x) |\vec H_{\mu}|^2 + \langle d \phi(x),  (T_x^{\perp} \mu) \vec{H}_\mu(x) \rangle) \, d\mu(x),  
\end{equation*}
where $T_x\mu$ denotes the projection onto the approximate tangent plane of $\mu$ at $x$, and $(T_x^{\perp}\mu)$ the orthogonal one. 
If any of the above assumptions fails, set $\mathcal{B}(\mu, \phi) := - \infty$. We will drop the subscript $\mu$ when it is clear from context. We are now able to give the definition of a Brakke flow. 

\begin{definition}\label{brakke.def}
	A family of $k$-rectifiable Radon measures $(\mu_t)_{t \geq 0}$ defines a \emph{$k$-Brakke flow} if for each $\phi \in C^2_c(M, \mathbb{R}_{\geq 0})$ and each $t \geq 0$ there holds
	\begin{equation}\label{brakke.ineq}
		\overline{D}_t \mu_t(\phi) \leq \mathcal{B}(\mu_t, \phi), 
	\end{equation}
	where $\overline{D}_t$ denotes the upper derivative at time $t$, i.e.,
	\begin{equation*}
		\overline{D}_t f := \limsup_{s \to t} \frac{f(s) - f(t)}{s - t}. 
	\end{equation*} 
\end{definition}
In the present paper we will use the following alternative definition:
for each $0\le s<t$ and $\phi \in C^2(M, \mathbb{R}_{\geq 0})$, we require that
\begin{equation}\label{brakke.ineq.bis}
	\mu_t(\phi) - \mu_s(\phi) \leq \int_{s}^{t} \mathcal{B}(\mu_{\tau}, \phi) \, d\tau
\end{equation}
and
\begin{equation}\label{brakke.sum}
    \int_0^t|\mathcal{B}(\mu_{\tau}, \phi)|\,d\tau<\infty.
\end{equation}
The latter definition implies \cref{brakke.ineq} for integral Brakke flows \cite{Lahiri}.

Though we will not need it in our proof of \cref{thmBrakke}, we note that Ambrosio and Soner have also extended the notion of Brakke flow from families of varifolds to families of \emph{generalized $k$-varifolds}, i.e., Radon measures on the bundle
\begin{equation*}
	A_{n, k}(M) := \{S \in \Sym(TM) : - ng \leq S \leq g, \, \tr(S) \geq k \},
\end{equation*} 
and some of our arguments---e.g., concerning rectifiability---admit alternative approaches by invoking the machinery of generalized Brakke flows laid out in \cite{AmbrosioSoner}. While we opt to avoid generalized Brakke flows in the proof of \cref{thmBrakke} to keep the analysis relatively self-contained, we do employ the theory of generalized varifolds at a few points in the proof, to simplify some arguments.

We next recall the concept of ``enhanced flow" introduced by Ilmanen in \cite{EllipticReg}, giving extra structure to the measure-theoretic Brakke flows; here we assume some familiarity with the theory of currents, as treated in \cite{SimonGMT} or \cite{Federer}. 
For an integral current $T \in \mathbf{I}_{\text{loc}}^{k + 1}( M\times[0, \infty))$, and for a Borel set $B \subseteq [0, \infty)$, we will denote by $T_B$ the current obtained by restricting $T$ to $ M\times B$, i.e., $T \mres ( M\times B )$; moreover, for $t\in [0,\infty)$, we will denote by $T_t\in {\bf I}^k_{\text{loc}}(M)$ the slice at time $t$, viewed as a $k$-current on $M$, so that
$$T_t=(-1)^{n+1}\pi_*[\partial (T_{[t,\infty)})].$$
\begin{definition} \label{defEnh}
	Consider an integral $k$-current $T_0 \in \mathbf{I}^{k}(M)$ with $\partial T_0=0$. We say that $(T, (\mu_t)_{t \geq 0})$ is an \emph{enhanced motion} with initial condition $T_0$ if the following hold:
	\begin{enumerate}[(i)]
		\item we have $T \in \mathbf{I}_{\text{loc}}^{k + 1}( M\times[0, \infty))$ and $(-1)^{n+1}\partial T =  T_0\times\{0\}$; 
		\item the measure $B\mapsto\mathbb{M}(T_B)$ for sets $B\subseteq \mathbb{R}$ is locally finite and absolutely continuous with respect to the Lebesgue measure $\mathcal{L}^1$;  
		\item the measures $(\mu_t)_{t \geq 0}$ define a Brakke flow; 
		\item at time $0$, $|T_0|=\mu_{0}$; 
		\item we have $|T_t| \leq \mu_t$ for each $t \geq 0 $. 
	\end{enumerate}
	In the above, $T$ is called the \textit{undercurrent} and $(\mu_t)_{t \geq 0}$ is the \textit{overflow}. 
\end{definition}

\begin{remark}
	Note that, because of the inequality $|T_t|\le\mu_t$, mass cannot arbitrarily disappear, even though there might be sudden mass loss in some cases. Also, the discrepancy between the two measures $\mu_t$ and $|T_t|$ implies that an enhanced motion is not necessarily an enhanced motion with respect to later starting times. 
\end{remark}

\begin{remark}
	Note that the slices $T_t\in {\bf I}^k$ must lie in the same homology class in $H_k(M;\mathbb{Z})$ for all $t\in [0,\infty)$, and as a consequence we see that an enhanced motion with homologically nontrivial initial data $T_0$ never vanishes.
\end{remark}

\section{Discrepancy bounds and Huisken monotonicity formula} \label{secEpsIneq}

Following \cite[Section~4]{PigatiStern}, we define the stress-energy tensor
associated to a pair $(u_t, \nabla_t)$ of solutions of \cref{gf} by
$$\mathcal{T}^{\epsilon}_t:=e^{\epsilon}(u_t,\nabla_t)g-2\nabla_t u_t^*\nabla_t u_t-2\epsilon^2\omega_t^*\omega_t,$$ 
where we let $$(\nabla_t u_t^*\nabla_t u_t)(e_i,e_j):=\ang{(\nabla_t)_{e_i}u_t,(\nabla_t)_{e_j}u_t}$$
and $$\omega_t^*\omega_t(e_i,e_j):=\sum_{k=1}^n\omega_t(e_i,e_k)\omega_t(e_j,e_k),$$ with respect to a local orthonormal basis $(e_i)_{i=1}^n$ of $TM$. 
As in \cite[Section~6]{gammaConv}, we record the identity
\begin{align}\label{div.id}
	\begin{aligned}
		\operatorname{div}(\mathcal{T}^{\epsilon}_t)&=2\langle \nabla_t u_t,\nabla_t^*\nabla_t u_t\rangle+d\frac{W(u_t)}{\epsilon^2}
		+2\omega_t(\langle iu_t,\nabla_t u_t\rangle,\cdot)-2\epsilon^2\omega_t(d^*\omega_t,\cdot)\\
		&=-2\langle \nabla_t u_t,\dot{u}_t\rangle-2\epsilon^2 \omega_t(\cdot,\dot{\alpha}_t),
	\end{aligned}
\end{align}
where the divergence is understood as $\sum_{i=1}^n (D_{e_i}\mathcal{T}^{\epsilon}_t)(e_i, \cdot)$, for $D$ the Levi-Civita connection on $M$, and the second equality follows from the gradient flow equations \cref{gf}.
Also, for $\phi \in C^1(M)$ we have
\begin{align} \label{epsBrakke}
	\begin{aligned}
		&\frac{d}{dt} \int_{M} \phi e^\epsilon(u_t, \nabla_t) \\
		&= \int_M \phi\de_t e^\epsilon(u_t, \nabla_t) \\
		&= \int_M \phi[\ang{\nabla_t\dot u_t-i\dot\alpha_t u_t,\nabla_t u_t}+2\epsilon^2\ang{d\dot\alpha_t,\omega_t}
		-\epsilon^{-2}(1-|u_t|^2)\ang{\dot u_t,u_t}] \\
		&= -2\int_M [\phi (|\dot u_t|^2+\epsilon^2|\dot\alpha_t|^2)+\ang{\nabla_t u_t(d\phi),\dot u_t}+\epsilon^2\omega_t(d\phi,\dot\alpha_t)] \\
		&= \int_M [- 2 \phi (|\dot u_t|^2+\epsilon^2|\dot\alpha_t|^2) + \divergence(\mathcal{T}^\epsilon_t)(d\phi)],
	\end{aligned}
\end{align}
where we used \cref{gf} in the penultimate equality and \cref{div.id} in the last one.
This can be seen as an $\epsilon$-version of Brakke's inequality: the first term on the right-hand side corresponds to the so-called \emph{shrinking term}, whereas the second one corresponds to the \emph{transport term} (cf.\ \cite[Section~2]{Ilmanen} for the analogous step in the Allen--Cahn setting).

\subsection{Coarse bounds for discrepancy and energy density} 
As discussed in \cite[Section~6]{gammaConv}, the key to obtaining a Huisken-type monotonicity result for the system \cref{gf} is obtaining suitable bounds on the ``discrepancy"
\begin{equation}\label{discr}
	\xi_t:=\epsilon|\omega_t|-\frac{1-|u_t|^2}{2\epsilon}.
\end{equation}
between (the square roots of) the curvature and potential components of the energy density $e^{\epsilon}(u_t,\nabla_t)$, analogous to the role of the discrepancy $\frac{\epsilon}{2}|du|^2-\frac{W(u)}{\epsilon}$ between Dirichlet and potential terms in the Allen--Cahn setting. In what follows, we recall the coarse bounds for $\xi_t$ obtained in \cite[Section~6]{gammaConv}, before improving these to a uniform upper bound in \cref{better.disc.bd} below.

Let $K(t,x,y)$ be the heat kernel of $M$, defined for $t>0$ and $x,y\in M$, so that
$$-(\partial_t+d_x^*d_x)K(t,x,y)=0, \quad \lim_{t\to 0}K(t,\cdot,y)=\delta_y,$$
and recall the following asymptotics for the heat kernel on a compact Riemannian manifold.

\begin{proposition}\label{asympHK}
	Let $\Omega:=\{(x,y)\in M\times M:d(x,y)<\mz\operatorname{inj}(M)\}$.
	There exists a function $v_0:\Omega\to(0,\infty)$ with $v_0(x,x)=1$ and such that
	$$ (4\pi t)^{n/2}e^{d(x,y)^2/(4t)}K(t,x,y)\to v_0(x,y) $$
	uniformly on $\Omega$, as $t\to 0^+$.
	We also have the bound
	$$ |d_x K(t,x,y)|\le C(M)\left(\frac{1}{\sqrt{t}}+\frac{d(x,y)}{t}\right)K(t,x,y) $$
	for $0<t\le 1$.
\end{proposition}

A proof of the first part of the statement is given in \cite{Kannai}.
While the second part is also contained in \cite{Kannai}, it follows from the more elementary bound shown in \cite[Corollary~1.3]{Hamilton}.

As in \cite[Section~6]{gammaConv}, define
\begin{equation}\label{varphi}
	\varphi_t(x)=\varphi(x,t):=\int_MK(t,x,y)|\xi_0|(y)\,dy,
\end{equation}
which solves the heat equation $- (\partial_t + d^*d)\varphi = 0$ with initial condition $\varphi(0, x) = \vert \xi_0(x) \vert$, and set
\begin{equation}\label{psi}
	\psi_t(x)=\psi(x,t):=\int_0^t\int_M K(t-s,x,y)\frac{C_0}{2\epsilon}(1-|u_s|^2)(y)\,dy\,ds,
\end{equation}
which solves the inhomogeneous heat equation $-(\partial_t + d^*d)\psi=-\frac{C_0}{2\epsilon}(1-|u_t|^2)$ with zero initial condition. For a suitable choice of $C_0=C_0(M)$, it is shown in \cite[Section~6.1]{gammaConv}
that $\xi_t$ obeys the bounds
\begin{equation}\label{xi.psi}
	\xi_t\le e^{Ct}(\varphi_t+\psi_t)
\end{equation}
and
\begin{equation} \label{crudeDiscrepancy}
	\xi_t\leq Ce^{Ct}\left(\varphi_t+\epsilon^{\frac{1}{n-1}}+\epsilon^{\frac{2}{n-1}}\frac{1-|u_t|^2}{\epsilon}\right),
\end{equation}
for a constant $C=C(M,\Lambda)$ depending only on $M$ and an initial energy bound $E_{\epsilon}(u_t,\nabla_t)\leq\Lambda$.

In \cite[Section~6]{gammaConv}, we used these estimates to derive a coarse Huisken-type monotonicity formula, with the simple goal of ruling out energy concentration at a point for solutions of \cref{gf} at positive times; in particular, in \cite[Section~6.2]{gammaConv} we proved the following.

\begin{proposition}\cite[Proposition~6.3]{gammaConv}\label{ener.bd}
	Given $\Lambda,t_0>0$, for a solution $(u_t,\nabla_t)$ of the gradient flow equations with initial energy $E_\epsilon(u_0,\nabla_0)\le\Lambda$ we have
	$$ \int_{B_r(y)}e^\epsilon(u_t,\nabla_t)\le C(M,\Lambda,t_0)r^{n-2} $$
	for all $y\in M$, $r>\epsilon$, and $t\ge t_0>0$.
\end{proposition}

\begin{remark}
	As written, \cite[Proposition~6.3]{gammaConv} applies to radii $\epsilon<r\leq 1$ and times $t\ge 2$, but the statement for radii $r\geq 1$ follows immediately from the bound $E_{\epsilon}(u_t,\nabla_t)\leq \Lambda$, and the extension to times $0<t_0<2$ (with constants depending on $t_0$) is similarly trivial (for instance, it can be obtained from the $t_0=2$ case via parabolic rescaling).
\end{remark}

\subsection{Refined bounds and Huisken-type monotonicity formula}\label{better.disc.bd}
Using \cref{ener.bd}, we now improve \cref{crudeDiscrepancy} to a uniform upper bound for $\xi_t=\epsilon|\omega_t|-\frac{1-|u_t|^2}{2\epsilon}$.
\begin{proposition}\label{disc.refined}
	Given $\Lambda,t_0>0$, for a solution $(u_t,\nabla_t)$ of the gradient flow equations with initial energy $E_\epsilon(u_0,\nabla_0)\le\Lambda$ we have
	$$ \xi_t\le C(M,\Lambda,t_0) $$
	for all times $t\ge t_0>0$.
\end{proposition}

\begin{proof}
	Without loss of generality, we can assume that $t_0\le t\le 1$, since the case $t>1$ follows from this one by considering the translated family $(u_{\tau+s},\nabla_{\tau+s})_{s\geq 0}$ with initial data $(u_{\tau},\nabla_{\tau})$, where $\tau:=t-1$.
	
	Using \cref{xi.psi}, we have
	$$ \xi_t(x)\le C(M,\Lambda,t_0)+C(M,\Lambda)\int_0^t\int_M K(t-s,x,y)\frac{1-|u_s|^2(y)}{2\epsilon}\,dy\, ds. $$
	For $s\in[0,t_0/2]$, using \cref{asympHK}, we can bound
	$$K(t-s,x,y)\le C(M)(t-s)^{-n/2}\le C(M)t_0^{-n/2},$$
	so that
	$$ \int_0^{t_0/2}\int_M K(t-s,x,y)\frac{1-|u_s|^2(y)}{2\epsilon}\,dy\, ds
	\le \frac{C(M)}{t_0^{n/2}}\int_0^{t_0/2}\left(\int_M\sqrt{e_s^\epsilon}\right)\,ds
	\le \frac{C(M)\sqrt{\Lambda}}{t_0^{n/2-1}}, $$
	where we let $e_s^\epsilon=e^{\epsilon}(u_s,\nabla_s).$
	For the remaining times $s\in[t_0/2,t]$, let $\rho(s):=\sqrt{t-s}$ and recall that \cref{ener.bd}
	gives
	$$ \int_{B_r(x)}\sqrt{e_s^\epsilon}(y)\,dy\le\operatorname{vol}_g(B_r(x))^{1/2}\left(\int_{B_r(x)}e_s^\epsilon(y)\,dy\right)^{1/2}\le Cr^{n-1} $$
	for some constant $C=C(M,\Lambda,t_0)$, provided that $r>\epsilon$.
	For these times, the previous inner integral is then bounded by
	\begin{align*}
		\int_M K(t-s,x,y)\sqrt{e_s^\epsilon}(y)\,dy
		&\le C\int_M \rho^{-n}e^{-d(x,y)^2/(4\rho^2)}\sqrt{e_s^\epsilon}(y)\,dy \\
		&\le C\rho^{-n}\int_\rho^{\infty}\left(-\frac{d}{dr}e^{-r^2/(4\rho^2)}\right)\left(\int_{B_r(x)}\sqrt{e_s^\epsilon}(y)dy\right)\,dr \\
		&\le C\rho^{-n}\int_\rho^\infty\frac{r}{\rho^2}e^{-r^2/(4\rho^2)}r^{n-1}\,dr,
	\end{align*}
	and making the substitution $\lambda=r/\rho$ in the final integral, we see that
	\begin{align*}
		\int_M K(t-s,x,y)\sqrt{e_s^\epsilon}(y)\,dy
		&\le C\rho^{-1}\int_1^\infty \lambda^n e^{-\lambda^2/4}\,d\lambda
		= \frac{C}{\sqrt{t-s}},
	\end{align*}
	provided that $t-s=\rho(s)^2>\epsilon^2$.
	Finally, for $s\in[t-\epsilon^2,t]$ we use the trivial bound
	$$ \int_M K(t-s,x,y)\frac{1-|u_s|^2(y)}{2\epsilon}\,dy\le\frac{1}{2\epsilon}, $$
	which gives a term bounded by $C\epsilon$ when integrated over this time interval.
	Combining the preceding estimates, we see that
	$$ \xi_t\le C+C\int_{t_0/2}^t\frac{ds}{\sqrt{t-s}}\le C $$
	for some constant $C=C(M,\Lambda,t_0)$.
\end{proof}

With little additional effort, we can now show a useful pointwise bound for the full energy density.

\begin{proposition} \label{prop.nabla.u.bd}
	Given $\Lambda,t_0>0$, there exists a constant $C(M,\Lambda,t_0)$ such that
	\begin{align}\label{nabla.u.bd}
		&e^\epsilon(u_t,\nabla_t)
		\le C\frac{(1-|u_\epsilon|^2)^2}{\epsilon^2}+C
	\end{align}
	whenever $t\ge t_0>0$, for a solution $(u_t,\nabla_t)$ of \cref{gf} with $E_\epsilon(u_0,\nabla_0)\le\Lambda$.
\end{proposition}

\begin{proof}
	Again we can assume $t \le 1$; for the remainder of the proof, for the sake of readability we drop the subscript $t$ in our notation.
	Since $\epsilon|\omega|=\frac{1-|u|^2}{2\epsilon}+\xi$ and $\xi\le C$, it is enough to obtain the bound
	$$ |\nabla u|\le C\frac{1-|u|^2}{\epsilon}+C. $$
	Let $$w:=|\nabla u|-\lambda\frac{1-|u|^2}{\epsilon},$$
	where $\lambda$ is a constant that will be chosen later. Let $\tilde\varphi_t\ge 0$ be the solution to the heat equation with $\tilde\varphi_0=|w_0|$ and let $\eta:=w-\tilde\varphi-\lambda\varphi$, where $\varphi$ is defined in \cref{varphi}. Recall now the B\"ochner formula \cref{NablaSquared.boch} for $\vert \nabla u \vert^2$; using the bound $|\psi(u,\nabla)|\le|\nabla u|^2$, whose simple proof can be found in \cite[Section~2]{PigatiStern}, we easily deduce the weak subequation
	\begin{align*}
		-(\de_t+d^*d)|\nabla u|
		&\ge\frac{3|u|^2-1}{2\epsilon^2}|\nabla u|-2|\omega||\nabla u|-C(M)|\nabla u|. 
	\end{align*}
	Using also \cref{Wboch}, we obtain
	\begin{align*}
		&-(\de_t + d^*d)\eta
		=-(\de_t + d^*d)w
		\ge\frac{|u|^2}{\epsilon^2}w+|\nabla u|\left(\frac{2\lambda}{\epsilon}|\nabla u|-\frac{1-|u|^2}{2\epsilon^2}-2|\omega|-C(M)\right).
	\end{align*}
	Since $\eta_0\le 0$, at a positive maximum for $\eta$ on $ M\times[0,1]$ we have
	$|\nabla u| \ge w\ge\eta>0$ and thus, by the maximum principle, $$\frac{2\lambda}{\epsilon}|\nabla u|\le\frac{1-|u|^2}{2\epsilon^2}+2|\omega|+C(M).$$ Assuming $\lambda\ge 1$, this implies the gradient bound 
	\begin{align*}
		|\nabla u|
		& \le \frac{1-|u|^2}{4\epsilon}+\epsilon|\omega|+C(M)\epsilon \\
		& \le \frac{1-|u|^2}{\epsilon}+\xi+C(M) \\
		& \le C_1\frac{1-|u|^2}{\epsilon}+C_1\varphi+C_1
	\end{align*}
	at this maximum point for $\eta$,
	for some $C_1(M,\Lambda)$, where the last inequality follows from \cref{crudeDiscrepancy}. Consequently, choosing $\lambda:=C_1$, at this maximum point we have
	$$\eta\le |\nabla u|-\lambda\frac{1-|u|^2}{\epsilon}-\lambda\varphi\le C,$$ 
	and therefore $\eta\le C$ on all of $ M\times[0,1]$. In particular, it follows that $$|\nabla u|\le C\frac{1-|u|^2}{\epsilon}+\tilde\varphi+C\varphi + C.$$ To conclude, we observe that pointwise bounds $\tilde{\varphi},\varphi\le C(M,\Lambda,t_0)$ for $t\ge t_0>0$ follow immediately from the definitions of $\varphi$ and $\tilde{\varphi}$, and standard propeties of the heat kernel.
\end{proof}

As an immediate consequence of the preceding proposition, we see that $\int_{B_r(y)}e^\epsilon(u_t,\nabla_t)\le C\epsilon^{-2}r^n$, so the statement of \cref{ener.bd} actually holds for all radii $r\leq \epsilon$ as well.

\begin{corollary}\label{ener.bd.bis}
	For $t\ge t_0>0$, $y\in M$, and $r>0$, we have
	$$ \int_{B_r(y)}e^\epsilon(u_t,\nabla_t)\le C(M,\Lambda,t_0)r^{n-2}. $$
\end{corollary}

With the preceding estimates in hand, we can obtain a Huisken-type monotonicity formula that sharpens the preliminary one obtained in \cite[Section~6.2]{gammaConv}. In what follows, given $T>0$, let $h_t(x)=h(x,t)$ be a positive solution of the \emph{backward} heat equation $\partial_th=d^*dh$ on $ M\times[0,T)$, with $\int_Mh_t(x)\,dx=1$; specifically, we take
\begin{equation*}
	h_t(x) = h(x,t) := K(T - t, x, y).
\end{equation*}

\begin{proposition}\label{mono.prop}
	Given $\Lambda,t_0,t_1>0$ with $t_0<t_1$, and given $(u_t,\nabla_t)$ solving the gradient flow equations with $E_\epsilon(u_0,\nabla_0)\le\Lambda$, there exists a constant $C(M,\Lambda,t_0,t_1)$ such that
	$$ \frac{d}{ds}\left(e^{C\sqrt{T-s}}(T-s)\int_M h_se^{\epsilon}(u_s,\nabla_s)\right)\le\frac{C}{\sqrt{T-s}} $$
	for all $t_0\le s<T\le t_1$. In particular, we have
	$$ e^{C\sqrt{T-s}}\int_M h_se^{\epsilon}(u_s,\nabla_s)\ge e^{C\sqrt{T-t}}(T-t)\int_M h_te^{\epsilon}(u_t,\nabla_t)-2C\sqrt{t-s} $$
	for all $t_0\le s\le t<T\le t_1$, for the same constant $C$.
\end{proposition}

\begin{proof}
	
	As in \cite[Section~6.2]{gammaConv}, we begin by setting $e_t:=e^\epsilon(u_t,\nabla_t)$ and introducing the function
	\begin{equation*}
		\Phi_h(t):=\int_M h_t e_t. 
	\end{equation*}
	Computing exactly as in \cite[Section~6.2]{gammaConv}, we obtain the estimate
	\begin{equation}\label{phi'est}
		\Phi_h'(s)\leq \int_M\left(\frac{h_s}{T-s}+C+Ch_s\log(B/(T-s)^{n/2})\right)(|\nabla u|^2+2\epsilon^2|\omega|^2),
	\end{equation}
	for some constants $C(M),B(M)$, where we have supressed the subscript $t$ in $u_t$, $\nabla_t$, and $\omega_t$ for simplicity of notation. Writing
	\begin{align*}
		|\nabla u|^2+2\epsilon^2|\omega|^2 &= e_s+ \epsilon^2|\omega|^2-\frac{(1-|u|^2)^2}{4\epsilon^2}\\
		&= e_s+\xi_s\left(\epsilon |\omega|+\frac{(1-|u|^2)}{2\epsilon}\right)\\
		&\leq e_s+2(\xi_s)^+\sqrt{e_s},
	\end{align*}
	we can use \cref{disc.refined} in the preceding estimate for $\Phi_h'(s)$ to see that
	\begin{align*}
		\Phi_h'(s) &\leq\int_M\left(\frac{h_s}{T-s}+C+Ch_s\log(B/(T-s)^{n/2})\right)(e_s+2(\xi_s)^+\sqrt{e_s})\\
		&\leq\frac{1}{T-s}\Phi_h(s)+\frac{C}{T-s}\Phi_h(s)^{1/2} +C+C\log(B/(T-s)^{n/2})\Phi_h(s) 
	\end{align*}
	for $s\ge t_0$ and another constant $C(M,\Lambda,t_0)$.

In particular, we have
$$ \Phi_h'(s) \leq \frac{1}{T-s}\Phi_h(s)+\frac{C_2}{T-s}\Phi_h(s)^{1/2}+C_2+\frac{C_2}{2\sqrt{T-s}}\Phi_h(s) $$
for some $C_2(M,\Lambda,t_0,t_1)$.
Setting now
$$ \Psi_h(s):=(T-s)e^{C_2\sqrt{T-s}}\Phi_h(s), $$
we get again the differential inequality
$$ \Psi_h'(s)\le C(T-s)^{-1/2}\Psi_h(s)^{1/2}+C. $$
The same argument used in the proof of \cref{disc.refined} gives
$$ (T-s)\Phi_h(s)\le C(T-s)\int_M (T-s)^{-n/2}e^{-d(x,y)^2/[4(T-s)]}e_s(x)\,dx\le C, $$
provided that $T-s>\epsilon^2$. On the other hand, if $T-s\le\epsilon^2$ we reach the same conclusion using the pointwise bound $e_s\le C\epsilon^{-2}$ from \cref{prop.nabla.u.bd}, which gives $\Phi_h(s)\le C\epsilon^{-2}$.
We then obtain
$$ \Psi_h(s)\le C. $$
Hence, the previous differential inequality becomes
$$ \Psi_h'(s)\le C(T-s)^{-1/2} $$
for some constant $C(M,\Lambda,t_0,t_1)$.
\end{proof}

\begin{remark}\label{mono.quantitative}
The inequality \cref{phi'est} is obtained in \cite[Section~6.2]{gammaConv} from a stronger inequality, which includes an additional term of the form
\begin{align}\label{shrinker.term}
	-2\int_M h_t[|\dot u+h_t^{-1}(\nabla_t)_{dh_t} u_t|^2+\epsilon^2|\dot\alpha+h_t^{-1}\iota_{dh_t}\omega_t|^2]
\end{align}
on the right-hand side.
If we do not drop \cref{shrinker.term}, we get the more quantitative bound
\begin{align*}
	&\frac{d}{ds}\left(e^{C\sqrt{T-s}}(T-s)\int_M h_se_s\right) \\
	&\le
	-2(T-s)\int_M h_s[|\dot u+h_s^{-1}(\nabla_s)_{dh_s} u_s|^2+\epsilon^2|\dot\alpha+h_s^{-1}\iota_{dh_s}\omega_s|^2]+\frac{C}{\sqrt{T-s}}.
\end{align*}
\end{remark}

\begin{remark}
Note that in $\mathbb{R}^n$, for initial data satisfying $\epsilon|F_{\nabla}|\leq \frac{1-|u|^2}{\epsilon}$, one can obtain a cleaner version of the above Huisken monotonicity formula. More precisely, letting
\begin{equation*}
\Psi_{T,p}^\epsilon(t):=(4\pi)^{-n/2}(T-t)^{-(n-2)/2}\int_{\R^n}
e^{-\frac{|x-p|^2}{4(T-t)}}\,d\mu_t^{\epsilon},
\end{equation*}
assuming that $\xi_t\le 0$ and that $e_t$ vanishes rapidly at infinity (for each $t>0$), from the previous computations we get 
\begin{align*}
&\frac{d}{dt}\Psi_{T,p}^\epsilon(t) \\
&=-\int_{\R^n} \left[ h |\xi|\left(\epsilon \vert \omega \vert + \frac{1-|u|^2}{2\epsilon}\right)+(T-t)(|\dot{u}+\nabla_{\frac{dh}{h}}u|^2+\epsilon^2|\dot{\alpha}+\iota_{\frac{dh}{h}}\omega|^2)\right],
\end{align*}
where we dropped the subscript $t$ from each term.
\end{remark}

\subsection{Passing measures to the limit} \label{secLimMeasure}
We conclude this section by explaining how to obtain a limiting family of measures $(\mu_t)_{t}$ from the energy measures $(\mu_t^\epsilon)_t$ as $\epsilon\to 0$. 

Following \cite[Section~5]{Ilmanen}, we start by proving that $\mu^\epsilon_t$ satisfies a ``semi-decreasing'' property with constants independent of $\epsilon$. By \cref{epsBrakke}, we have
\begin{align} \label{semidec}
\begin{split}
\frac{d}{dt} \mu_t^\epsilon(\phi) & =  -2 \int_M \phi ( |\dot{u}|^2+\epsilon^2|\dot{\alpha}|^2 ) - 2\int_M ( \langle \nabla_{d\phi}u,\dot{u}\rangle+\epsilon^2\omega(d\phi,\dot{\alpha})) \\
& \le - \int_M (\phi \left\vert \dot{u} + \phi^{-1}\nabla_{d\phi}u \right\vert^2+\epsilon^2 \phi \left\vert \dot{\alpha}+\phi^{-1}\iota_{d\phi}\omega \right\vert^2) \\
& \quad + \int_M \phi^{-1}( |\nabla_{d\phi}u|^2+\epsilon^2|\iota_{d\phi}d\alpha|^2 )  \\
& \leq \int_M \phi^{-1}(|\nabla_{d\phi}u|^2+\epsilon^2|\iota_{d\phi}d\alpha|^2) \\ 
& \leq \sup_{\{\phi > 0\}} \frac{\vert d\phi \vert^2}{\phi} \mu_{t}^{\epsilon}(\{ \phi > 0 \}) \\
& \leq C(\phi),  
\end{split} 
\end{align}
for $\phi \in C^2(M, \mathbb{R}_{\ge0})$, where in the last line we used the interpolation estimate 
\begin{equation} \label{C1phi}
\sup_{\{\phi > 0\}} \frac{\vert d\phi \vert^2}{\phi} \leq 2\max \vert D^2 \phi \vert,
\end{equation}
whose proof can be found in \cite[Lemma~6.6]{EllipticReg}. Consequently, 
$$t\mapsto \mu_t^\epsilon (\phi) - C(\phi)t$$ 
is a  non-increasing function; in the literature, a family of measures satisfying an inequality of this type is someties called \textit{semi-decreasing}. Crucially, the constant $C(\phi)$ is independent of $\epsilon$ in this case.

Now, as in \cite[Section~5]{Ilmanen}, we choose a countable and dense set $B_1 \subset [0, \infty)$. By the mass bounds we have for the measures $\mu_t^\epsilon$ and the compactness properties of Radon measures, we may select a subsequence $(\epsilon_j)_{j\in\N}$ and measures $(\mu_t)_{t \in B_1}$ such that 
\begin{equation}
\mu_t^{\epsilon_j} \rightharpoonup^* \mu_t \quad \text{for all $t \in B_1$}.
\end{equation}
Note that the family $B_1\ni t\mapsto \mu_t$ then inherits the semi-decreasing property with the same constants $C(\phi);$ i.e.,
$$B_1\ni t\mapsto \mu_t(\phi)-C(\phi)t$$
is a non-increasing function for every $\phi\in C^2(M,\mathbb{R}_{\ge0}).$

Next, let $\{\phi_k\}_{k \in \mathbb{N}}$ be a countable and dense subset of $C^2(M, \mathbb{R}_{\ge0})$. By the semi-decreasing property for the family $(\mu_t)_{t\in B_1}$, there is a set $B_2 \subseteq [0, \infty)$, whose complement is at most countable, such that for $t \in B_2$ and all $k \in \mathbb{N}$ we have that $\mu_s(\phi_k)$ is continuous at $t$ as a function of $s \in B_1$. For any fixed $t \in B_2$, we can find a further subsequence $(\mu^{\epsilon_{j_\ell}}_t)$ converging to a limit $\mu_t$.

Again by the semi-decreasing property, we have that $(\mu_s(\phi_k))_{s \in B_1 \cup \{t\}}$ is continuous at $t$ for each $k \in \mathbb{N}$. Density of $\{\phi_k\}_{k \in \mathbb{N}}$ implies that the limiting measure $\mu_t$ is in fact uniquely determined by $(\mu_s)_{s \in B_1}$. Thus, the full sequence converges: $\mu_t^{\epsilon_j} \rightarrow \mu_t$ and, consequently, the limiting measures $\mu_t$ are defined for each $t \in B_2$. To extend it to the whole of $[0, \infty)$ we perform a further diagonal sequence argument, identical to those in \cite[p.~433]{Ilmanen} or \cite[Section~7]{EllipticReg}. Summarizing, we have proved the following. 

\begin{proposition} \label{CorLimMeas}
There exists a subsequence $(\epsilon_j)_{j\in\N}$ with $\epsilon_j\to 0$ and a family of Radon measures $(\mu_t)_{t\ge0}$ such that $\mu_t^{\epsilon_j}$ weak-$*$ converge to $\mu_t$ for all $t \geq 0$, as Radon measures. 
\end{proposition}

\section{Clearing-out and exponential decay}\label{decay.sec}

We prove in this section various ``clearing-out'' or energy quantization results, and establish exponential decay of the energy away from the zero set, which are crucial ingredients in the proof of \cref{thmBrakke}, analogous in spirit to some results of \cite[Section~6]{Ilmanen} in the Allen--Cahn setting. 

\begin{proposition}[clearing-out for values]\label{clearing.values}
	Given $\delta\in(0,1)$ there exist positive constants $\eta_V(M,\Lambda,\delta,t_0)$ and $c_V(M,\Lambda,\delta,t_0)$ such that, given $r\in(\epsilon,\mz\inj]$ and $t\ge t_0>0$, the following holds:
	if $E_\epsilon(u_0,\nabla_0)\le\Lambda$ and $\mu_t^\epsilon(B_r(p))\le\eta_V r^{n-2}$, then
	\begin{align*}
		&|u_\epsilon|^2(p,t+c_Vr^2)
		\ge 1-\delta.
	\end{align*}
\end{proposition}

\begin{proof}
Let $c_1>0$ be a constant to be determined later.
From \cref{nabla.u.bd} we get $$|\nabla_\epsilon u_\epsilon|(\cdot,t+c_1r^2)\le \frac{C(M,\Lambda,t_0)}{\epsilon}.$$
Hence, assuming $|u_\epsilon|^2(p,t+c_1r^2)<1-\delta$, using the bound $|d|u|^2|\le 2|\nabla u|$ we get
\begin{align*}
		&|u_\epsilon|^2(x,t+c_1r^2)<1-\frac{\delta}{2}
\end{align*}
for all $x\in B_{s}(p)$, with $s:=c_2\delta\epsilon$ for $c_2(M,\Lambda,t_0)$ sufficiently small. 
This implies that $\mu_{t+c_1r^2}^\epsilon(B_{s}(p))\ge c(M,\Lambda,\delta,t_0)\epsilon^{-2}s^{n}=c(M,\Lambda,\delta,t_0)s^{n-2}$.
Thus, in view of the asymptotics of \cref{asympHK}, setting
$$\rho_{p,\tau}(x,t):=(\tau-t)K(\tau-t,p,x),$$
we have
\begin{align*}
		&\int \rho_{p,t+c_1r^2+s^2}(x,t)\,d\mu_{t+c_1r^2}^\epsilon(x)
		\ge c(M,\Lambda,\delta,t_0).
\end{align*}
Together with \cref{mono.prop} (as usual, we can assume without loss of generality that $t\le 1$), this gives
\begin{align*}
		&\int \rho_{p,t+c_1r^2+s^2}(x,t)\,d\mu_t^\epsilon(x)
		\ge c_3(M,\Lambda,\delta,t_0)
	\end{align*}
if $r,s$ are small enough.
Also, by \cref{asympHK}, since $c_1r^2+s^2\le c_1r^2+(c_2r)^2\le 2c_1r^2$, assuming $c_2^2\le c_1$, we can now fix $c_1$ such that
	\begin{align*}
		&\int_{M\setminus B_r(p)}\rho_{p,t+c_1r^2+s^2}(x,t)\,d\mu_t^\epsilon(x)
		\le C\sum_{k=1}^\infty c_1^{1-n/2}e^{-(k-1)^2r^2/(2c_1r^2)}\frac{\mu_t^\epsilon(B_{kr}(p))}{r^{n-2}}
		\le \frac{c_3}{2},
\end{align*}
where we used \cref{ener.bd.bis}.
This gives
\begin{align*}
		&\int_{B_r(p)} \rho_{p,t+c_1r^2+s^2}(x,t)\,d\mu_t^\epsilon(x)
		\ge\frac{c_3}{2}.
\end{align*}
Since $\rho_{p,t+c_1r^2+s^2}(x,t)\le c_4 r^{2-n}$ on $B_r(p)$,
we must have $\mu_t^\epsilon(B_r(p))\ge \frac{c_3}{2c_4}r^{n-2}$.
The claim follows with $c_{V}:=c_1$ and $\eta_V:=\frac{c_3}{4c_4}$.
\end{proof}
We are now able to prove exponential decay away from the zero set up to an error term, generalizing \cite[Proposition~5.3]{PigatiStern} from the stationary case.

\begin{lemma}[Exponential decay]\label{decay}
	There exist constants $a_D(M)>0$, $\beta_D(M,\Lambda)\in(0,1)$ and $C(M,\Lambda,t_0)$ with the following property.
	Given $t\ge t_0>0$ and $p\in M$, we have
	\begin{align*}
		&e^\epsilon(p,t)
		\le \frac{C}{\epsilon^2}e^{-a_D r(p,t)/\epsilon}+C\epsilon^2,
	\end{align*}
	where $r(p,t)$ is the maximum value $r\in[0, \inj/2]$ such that $r^2\le\frac{t}{2}$ and $|u|^2\ge 1-\beta_D$ on $\bar B_r(p)\times [t-r^2,t]$. We set $r(p,t)=0$ if $|u|^2(p,t)<1-\beta_D$. 
\end{lemma}

\begin{proof}
	We first show that, for some positive constants $C=C(M,\Lambda)$ and $a=a(M)$,
	\begin{align}\label{pot.decay}
		&\frac{1-|u|^2(p,t)}{2\epsilon}
		\le \frac{C}{\epsilon}e^{-ar(p,t)/\epsilon}+C\epsilon.
	\end{align} 
	
	Let $r:=r(p,t)$; clearly we can assume that $r>0$, since otherwise the claim holds trivially with $C=\frac{1}{2}$. 
	Recall that
	\begin{align*}
		(\de_s-\Delta)\frac{1-|u|^2}{2}(x,s)
		&=|\nabla u|^2-\frac{|u|^2}{2\epsilon^2}(1-|u|^2) \\
		&\le \frac{C(1-|u|^2)-|u|^2}{2\epsilon^2}(1-|u|^2)+C
	\end{align*}
	by \cref{u.boch} and \cref{nabla.u.bd}. Thus, if we fix $\beta_D(M,\Lambda)$ small enough, we have
	\begin{align*}
		&(\de_s-\Delta)\frac{1-|u|^2}{2}
		\le-\frac{1-|u|^2}{4\epsilon^2}+C
	\end{align*}
	on $\bar B_r(p)\times[t-r^2,t]$ (since $|u|^2\ge 1-\beta_D$ on this set, by definition of $r$).
	Let
	\begin{align*}
		&\varphi(x,s):=\exp[(a/\epsilon)(d_p(x)^2+(t-s)+\epsilon^2)^{1/2}]
	\end{align*}
	and $f:=\frac{1-|u|^2}{2}-\lambda\varphi$, for constants $a,\lambda$ to be chosen later. Then it is easy to compute
	\begin{align*}
		&(\de_s-\Delta)\varphi(x,s) \\
		&\ge -\frac{a}{\epsilon}\left(\frac{1+\Delta d_p^2(x)}{2(d_p(x)^2+(t-s)+\epsilon^2)^{1/2}}
		+\frac{(a/\epsilon)d_p(x)^2}{d_p(x)^2+(t-s)+\epsilon^2}\right)\varphi(x,s) \\
		&\ge -\frac{a+C(M)a+a^2}{\epsilon^2}\varphi(x,s) \\
		&\ge -\frac{1}{2\epsilon^2}\varphi(x,s)
	\end{align*}
	once we fix $a=a(M)$ small enough.
	For $f$ we deduce
	\begin{align*}
		&(\de_s-\Delta)f
		\le -\frac{1}{2\epsilon^2}f+C.
	\end{align*}
	Choosing now $\lambda:=\beta_D e^{-ar/\epsilon}$, we get
	\begin{align*}
		&f(x,s)
		\le\frac{\beta_D}{2}-\lambda\varphi(x,s)\le\frac{\beta_D}{2}-\lambda e^{ar/\epsilon}
		<0
	\end{align*}
	for all $(x,s)$ on the parabolic boundary $( B_r(p)\times\{t-r^2\})\cup(\de B_r(p)\times[t-r^2,t])$, since for these points we have $1-|u|^2(x,t)\le\beta_D$ by definition of $r$.
	If $(x,s)$ is a positive maximum for $f$ on $\bar B_r(p)\times[t-r^2,t]$, we then obtain
	\begin{align*}
		&0
		\le -\frac{1}{2\epsilon^2}f(x,s)+C,
	\end{align*}
	and thus $f\le C\epsilon^2$ on $\bar B_r(p)\times[t-r^2,t]$.
	Evaluating at $(p,t)$ we get
	\begin{align*}
		&\frac{1-|u|^2(p,t)}{2}\le (\beta_D e^a)e^{-ar/\epsilon}+C\epsilon^2,
	\end{align*}
	as desired. 
 
 As in \cite[Section~3]{PigatiStern}, from \cref{omega.boch} and \cref{Wboch} we easily get the subequation
	\begin{align*}
		&(\de_t-\Delta)\xi
		\le -\frac{|u|^2}{\epsilon^2}\xi+C(M)\epsilon|\omega|
		\le -\frac{1}{2\epsilon^2}\xi+C\sqrt{e^\epsilon}
	\end{align*}
	for $\xi=\epsilon|\omega|-\frac{1-|u|^2}{2\epsilon}$ on the same region. We deduce that
	\begin{align*}
		&(\de_t-\Delta)(\xi-\lambda'\varphi)
		\le -\frac{1}{2\epsilon^2}(\xi-\lambda'\varphi)+C\sqrt{e^\epsilon}
	\end{align*}
	for $\lambda'>0$. By \cref{prop.nabla.u.bd}, we have the preliminary bound
	\begin{align}\label{e.prel}
		&\sqrt{e^\epsilon}(x,s)
		\le \frac{C}{\epsilon}.
	\end{align}  
	In particular, at an interior (i.e., not on the parabolic boundary) positive maximum $(x,s)$ for the function $\xi-\lambda'\varphi$, we have
    $$\xi-\lambda'\varphi\le C\epsilon^2\sqrt{e^\epsilon}\le C\epsilon.$$

	Choosing $\lambda':=\frac{C}{\epsilon}e^{-ar/\epsilon}$ for $C$ large enough,
    on the parabolic boundary of $\bar B_{r}(p)\times[t-r^2,t]$ we estimate $\xi$ using \cref{e.prel} and obtain
    $\xi-\lambda'\varphi<0$. Hence, $\xi-\lambda'\varphi\le C\epsilon$ on all the region. Evaluating at $(p,t)$, we obtain
	\begin{align*}
		\xi(p,t)
		& \le \lambda'\varphi(p,t)+C\epsilon
		\le \frac{C}{\epsilon}e^{-ar/\epsilon}
		+C\epsilon.
	\end{align*}
 
	Since $r(x,s)\ge r/2$ for all $(x,s)\in\bar B_{r/2}(p)\times[t-r^2/4,t]$, we deduce	
	\begin{align}\label{decay.int}
		&\frac{1-|u|^2}{2\epsilon}+\epsilon|\omega|
		\le \frac{C}{\epsilon}e^{-ar/(2\epsilon)}
		+C\epsilon
	\end{align}
	on this smaller set.
	Set now $w:=|\nabla u|-\frac{1-|u|^2}{\epsilon}$ and $\tilde w:=w-\lambda''\varphi$. Combining as in the proof of \cref{prop.nabla.u.bd} the B\"ochner formula \cref{NablaSquared.boch} for $\vert \nabla u \vert$ and \cref{Wboch} for $(1 - \vert u \vert^2)/\epsilon$, we obtain  
	\begin{align*}
		&(\de_t-\Delta)\tilde w
		\le -\frac{|u|^2}{\epsilon^2}(w-\lambda''\varphi)
		+|\nabla u|\left(-\frac{2}{\epsilon}|\nabla u|+\frac{1-|u|^2}{2\epsilon^2}+2|\omega|+C(M)\right).
	\end{align*}
	The maximum principle implies that, at an interior positive maximum for $\tilde w$, we have
	$$|\nabla u|\le\frac{1-|u|^2}{4\epsilon}+\epsilon|\omega|+C\epsilon,$$
    and thus, at the same maximum point, using \cref{decay.int} we have
    $$\tilde w\le \frac{C}{\epsilon}e^{-ar/(2\epsilon)}+C\epsilon.$$
    Hence, again choosing $\lambda'':=\frac{C}{\epsilon}e^{-ar/(2\epsilon)}$ so that $w-\lambda''\varphi<0$ on the parabolic boundary of $\bar B_{r/2}(p)\times[t-r^2/4,t]$ (where we bound $w$ with \cref{e.prel}), we obtain
    $$\tilde w(p,t)\le\max_{\bar B_{r/2}(p)\times[t-r^2/4,t]}\tilde w\le \frac{C}{\epsilon}e^{-ar/(2\epsilon)}+C\epsilon,$$
    and hence
	\begin{align*}
		w(p,t)
		&\le \lambda''\varphi(p,t)
		+\frac{C}{\epsilon}e^{-ar/(2\epsilon)}
		+C\epsilon \\
		&\le \frac{C}{\epsilon}e^{-ar/(2\epsilon)}
		+C\epsilon.
	\end{align*}
	Putting together the previous estimates, we obtain
	\begin{align*}
		&\sqrt{e^\epsilon}(p,t)
		\le \frac{C}{\epsilon}e^{-ar/(2\epsilon)}
		+C\epsilon,
	\end{align*}
	as desired. The statement follows with $a_D:=a$ and $C_D:=2C^2$.
\end{proof}

Combining \cref{clearing.values} with \cref{decay} and passing to the limit yields the following density lower bound for the limiting family of measures.

\begin{proposition}[clearing-out for support]\label{clearing.spt}
	There exist positive constants $\eta_S(M,\Lambda,t_0)$ and $c_S(M,\Lambda,t_0)$ such that, for $0<r\le\mz\inj$ and $t\ge t_0>0$, the following holds:
	if $\mu_t(B_r(p))\le\eta_S r^{n-2}$, then $(p,t+c_S r^2)\nin\spt(\mu)$
	and $p\nin\spt(\mu_{t+c_S r^2})$.
\end{proposition}

\begin{proof}
	Choosing $\eta_S<8^{2-n}\eta_V(M,\Lambda,\beta_D,t_0)$, we have
	\begin{align*}
		&\mu_t^\epsilon(B_{r/2}(p))
		\le 8^{2-n}\eta_V r^{n-2}
	\end{align*}
	for $\epsilon$ small enough, which gives
	\begin{align*}
		&\mu_t^\epsilon(B_s(x))
		\le \eta_V s^{n-2}
	\end{align*}
	for all $x\in \bar B_{r/4}(p)$ and all radii $s\in[\frac{r}{8},\frac{r}{4}]$.
	By \cref{clearing.values}, for $\epsilon$ small enough this implies that
	\begin{align*}
		&|u_\epsilon|^2
		\ge 1-\beta_D
		\quad\text{on }\bar B_{r/4}(p)\times[t+c_Vr^2/64,t+c_Vr^2/16].
	\end{align*}
	Hence, for all points $(x,s)$ with $s\in[t+c_V r^2/32,t+c_V r^2/16]$ and $x\in \bar B_{r/8}(p)$,
	we have $r(x,s)\ge\sqrt{c_V}\frac{r}{8}$ (we assumed here without loss of generality $c_V\le 1$). By \cref{decay} we get
	\begin{align*}
		&e^\epsilon
		\le \frac{C_D}{\epsilon^2}e^{-a_D \sqrt{c_V} r/(8\epsilon)}+C_D\epsilon^2
	\end{align*}
	on $\bar B_{r/8}(p)\times[t+c_V r^2/32,t+c_V r^2/16]$.
	Letting $\epsilon\to 0$, this readily implies that
	\begin{align*}
		&\mu( B_{r/8}(p)\times(t+c_V r^2/32,t+c_V r^2/16))
		=0,
	\end{align*}
	as well as $\mu_{t+c_V r^2/20}(B_{r/8}(p))=0$,
	so that the statement follows with $c_S:=c_V/20$.
\end{proof}

As a corollary, we have the following basic structural result (cf.\ \cite[Corollary~6.2]{Ilmanen}).

\begin{corollary}\label{clarity.spt}
	The measures $\mu_t$ provide the disintegration for $\mu$ on $ M\times(0,\infty)$, i.e.,
	\begin{align*}
		&\mu=\mu_t\otimes\mathcal{L}^1(t).
	\end{align*}
	Also, $\spt(\mu)=\overline{\bigcup_{t>0}(\spt(\mu_t)\times\{t\})}$. In particular, for all $t > 0$, we have $\spt \mu_t \subseteq (\spt \mu)_t$, the $t$-slice of $\spt\mu$.
\end{corollary}

\begin{proof}
	By the semi-decreasing property \cref{semidec}, the map $t\mapsto\mu_t$ is weakly measurable.
	Given $\chi\in C^0_c((0,\infty))$ and $\varphi\in C^0(M)$, we have
	\begin{align*}
		\int_{ M\times(0,\infty)}\chi(t)\varphi(x)\,d\mu(x,t)
		&=\lim_{\epsilon\to 0}\int_0^\infty\chi(t)\int_M\varphi(x)\,d\mu_t^\epsilon(x)\,dt \\
		&=\int_0^\infty\chi(t)\int_M\varphi(x)\,d\mu_t(x)\,dt
	\end{align*}
	by dominated convergence, since $\int_M\varphi\,d\mu_t^\epsilon$ is bounded by $\|\varphi\|_{C^0}\Lambda$ and converges to $\int_M\varphi\,d\mu_t$ as $\epsilon\to 0$. This proves the first assertion. Also, if $(p,t)\nin\spt(\mu)$ then there exists $0<r<t/2$ such that
	\begin{align*}
		&\mu_{t'}(B_{2r}(p))=0\quad\text{for a.e.\ }t'\in(t-r,t).
	\end{align*}
	For these values $t'$, this implies that $\mu_{t'}(B_s(q))=0$ for all $q\in B_r(p)$ and $0<s<r$. Hence, $q\nin\spt(\mu_{t'+c_S s^2})$ for all $q\in B_r(p)$ and $0<s<r$, with $c_S(M,\Lambda,t/2)$ given by \cref{clearing.spt}. This implies that $$(p,t)\nin\overline{\bigcup_{\tau>0}(\spt(\mu_\tau)\times\{\tau\})},$$ while the reverse implication is trivial. 
\end{proof}

\section{Generic tangent flows and rectifiability}\label{tf.sec}
The main goal of this section is to show that, for a.e.\ time $t>0$, the measure $\mu_t$ is \emph{rectifiable},
in the sense that it is supported on an $(n-2)$-rectifiable set. In the course of the argument we will also establish a lower bound on the $(n-2)$-density of $\mu_t$ at $\mu_t$-a.e.\ point, complementing the global upper bound given by \cref{ener.bd.bis}. 

\subsection{Parabolic vs Euclidean density}
Although not strictly needed, to simplify the exposition it is useful to show that (an averaged version of) the Euclidean density
$$ \Theta_{n-2}(\mu_t,p)=\lim_{r\to 0}\frac{\mu_t(B_r(p))}{\omega_{n-2}r^{n-2}} $$
exists at $\mu_t$-a.e.\ point $p$, for a.e.\ time $t>0$.
This fact is not obvious a priori, while a parabolic version
of this density, which we call \emph{parabolic density} or \emph{Gaussian density} (cf.\ \cite[Section~2.9]{White} or \cite[Section~II.0.1]{BethuelOrlandiSmets}), is easily seen to exist as a byproduct of monotonicity, as explained below.
We will then check that, at almost every point in the spacetime, an averaged version $\tilde\Theta_{n-2}(\mu_T,p)$ of the usual Euclidean density exists and agrees with the Gaussian one.

Given a point $p\in M$ and a time $T>0$, consider the backward heat kernel
$$h_t(x):=K(T-t,x,p).$$
Defining
$$\Psi_{T,p}^\epsilon(t):=e^{C'\sqrt{T-t}}(T-t)\int_M h_t\,d\mu_t^{\epsilon},$$
where $C'$ is the constant from \cref{mono.prop} (with $t_0:=T/2$ and $t_1:=T$), we have
\begin{equation}\label{mono.before.lim}
	(\Psi_{T,p}^\epsilon)'(t)
	\le \frac{C'}{\sqrt{T-t}}.
\end{equation}

For $t<T$, set
$$\Psi_{T,p}(t):=\lim_{\epsilon\to 0}\Psi_{T,p}^\epsilon(t)=e^{C'\sqrt{T-t}}(T-t)\int_M h_t\,d\mu_t,$$
and observe that, integrating \cref{mono.before.lim}, we have the limiting monotonicity
$$\Psi_{T,p}(s)\geq \Psi_{T,p}(t)-2C'\sqrt{t-s}\quad\text{for }s<t,$$
so in particular we can define the \emph{parabolic density} at $(p,T)$:
$$\Theta^P(\mu,p,T):=4\pi\lim_{t\to T^-}\Psi_{T,p}(t)=4\pi\lim_{t\to T^-}(T-t)\int_M h_t\,d\mu_t,$$
We also define the modified Euclidean density
$$ \tilde\Theta_{n-2}(\mu_T,p):=4\pi\lim_{t\to T^-}(T-t)\int_M h_t\,d\mu_T
=\lim_{r\to 0}4\pi r^2\int_M K(r^2,x,p)\,d\mu_T(x), $$
provided that the limit exists.
Using \cref{asympHK}, we see that
\begin{align}\label{mod.theta.bis} &\tilde\Theta_{n-2}(\mu_T,p)=\lim_{r\to 0}(4\pi r^2)^{-(n-2)/2}\int_M e^{-d(x,p)^2/(4r^2)}\,d\mu_T(x), \end{align}
and it is straightforward to check that $\tilde\Theta_{n-2}(\mu_T,p)=\Theta_{n-2}(\mu_T,p)$ whenever the latter exists.
Also, using monotonicity,
it is easy to see that the parabolic density dominates the modified Euclidean density: $\Theta^P(\mu,p,T)\geq \tilde\Theta_{n-2}(\mu_T,p)$ whenever the latter exists:
indeed, by definition we have
$$ \tilde\Theta_{n-2}(\mu_T,p)=4\pi\lim_{r\to 0}\Psi_{T+r^2,p}(T)\le 4\pi\lim_{r\to 0}\Psi_{T+r^2,p}(s)+C\sqrt{T-s} $$
for any $s<T$, where the inequality follows from the previous monotonicity. Since the last limit is just $\Psi_{T,p}(s)$, letting $s\to T^-$ we obtain the claim.

We next show that equality holds for $\mu$-a.e.\ $(p,T)\in  M\times(0,\infty)$.
Indeed, consider the set $G\subseteq \operatorname{spt}(\mu)$ of points $(p,T)\in M\times (0,\infty)$ in the support of $\mu$ such that
\begin{equation}
	M_{\nu}(p,T):=\sup_{r\in (0,1)}\frac{\nu(\bar B_r(p)\times [T-r^2,T+r^2])}{\mu(\bar B_{5r}(p)\times [T-25r^2,T+25r^2])}<\infty, 
\end{equation}
where as before $\nu^{\epsilon}$ denotes the measure
$$d\nu^{\epsilon}:=(|\dot{u}_t^{\epsilon}|^2+\epsilon^2|\dot{\alpha}_t^{\epsilon}|^2)\, \dvol_g\otimes dt$$
and $\nu=\lim_{\epsilon\to 0}\nu^{\epsilon}$ the (subsequential) limit. We extend $\mu$ and $\nu$ to $ M\times\R$, by letting them vanish on $ M\times(-\infty,0)$. 

\begin{proposition}\label{G.ae}
	In the positive spacetime $ M\times(0,\infty)$, $\mu$-almost every point belongs to $G$.
\end{proposition}

\begin{proof}
	Observe that we can write
	$$M_{\nu}(p,T)=\sup_{r\in (0,1)}\frac{\nu(\bar B^P_r(p,T))}{\mu(\bar B^P_{5r}(p,T))},$$
	where we denote by $B^P_r$ the $r$-ball with respect to the parabolic metric
	$$d_P((p,s),(q,t)):=\max\{d(p,q),\sqrt{|t-s|}\}.$$
	
	The proof now follows the usual lines of proof for the weak-$(1,1)$ boundedness of maximal functions: setting $$E:=\{x\in \spt(\mu) : M_{\nu}(x)=\infty\},$$
	where $x$ denotes a point in the spacetime,
	we see that for any $\delta>0$, we can find balls $\bar B^P_{r_x}(x)$ about every point $x\in E$ for which
	$$\nu(\bar B^P_{r_x}(x))>\frac{1}{\delta}\mu(\bar B^P_{5r_x}(x)).$$
	Applying the Vitali covering lemma, we can find an (at most) countable subcollection $\{x_1,x_2,\ldots\}\subset E$ for which the balls $\bar B^P_{r_j}(x_j)$ are disjoint, but the dilated balls $\bar B^P_{5r_j}(x_j)$ cover $E$, so that
	\begin{align*}
		&\nu( M\times\R)\geq\sum \nu(\bar B^p_{r_j}(x_j))
		\geq\frac{1}{\delta}\sum \mu(\bar B^P_{5r_j}(x_j))
		\geq\frac{1}{\delta}\mu(E).
	\end{align*}
	On the other hand, integrating \cref{epsBrakke} (with $\phi=1$), we have
	\begin{align}\label{nu.en}&2\nu^\epsilon( M\times[0,\infty))\le\Lambda,\end{align}
	and hence $\nu( M\times\R)=\nu(M\times [0,\infty))<\infty$. Since $\delta>0$ was arbitrary, we deduce that $\mu(E)=0$, as desired.
\end{proof}

Next, we show that for points in $G$, the parabolic and (modified) Euclidean densities $\Theta^P(\mu,p,T)$ and $\tilde\Theta_{n-2}(\mu_T,p)$ coincide.

\begin{proposition}\label{eq.dens} If $(p,T)\in G$, then $\tilde\Theta_{n-2}(\mu_T,p)$ exists and $\Theta^P(\mu,p,T)=\tilde\Theta_{n-2}(\mu_T,p)$.
\end{proposition}

\begin{proof} Writing
	$$\phi_r(x):=4\pi r^2 K(r^2,x,p),$$
	we have by definition
	$$\Theta^P(\mu,p,T)=\lim_{r\to 0}\mu_{T-r^2}(\phi_r)$$
	and
	$$\tilde\Theta_{n-2}(\mu_T,p)=\lim_{r\to 0}\mu_T(\phi_r)$$
	(provided that the limit exists). Hence, it suffices to show that
	$$\lim_{r\to 0}|\mu_{T-r^2}(\phi_r)-\mu_T(\phi_r)|=0$$
	for every $(p,T)\in G$.
	To this end, recall from \cref{epsBrakke} that
	\begin{align}\label{epsBrakke.bis}
		&\frac{d}{dt} \int_{M} \phi e^\epsilon(u_t, \nabla_t) 
		= -2\int_M [\phi (|\dot u_t|^2+\epsilon^2|\dot\alpha_t|^2)+\ang{\nabla_t u_t(d\phi),\dot u_t}+\epsilon^2\omega_t(d\phi,\dot\alpha_t)]
	\end{align}
	for any smooth function $\phi:M\to\R$, whence
	$$|\mu_s(\phi)-\mu_t(\phi)|\le C\int_{ M\times[s,t]}(|d\phi|\,d\mu+(|\phi|+|d\phi|)\,d\nu)$$
	for any two times $0\le s<t$ by Young's inequality. Using also \cref{asympHK}, it follows that
	\begin{align*}
		&|\mu_{T-r^2}(\phi_r)-\mu_T(\phi_r)| \\
		&\leq C\int_{ M\times[T-r^2,T]} (|d\phi_r|\,d\mu+(\phi_r+|d\phi_r|)\,d\nu)\\
		&\leq Cr^{2-n}\int_{ M\times[T-r^2,T]}\left(\frac{1}{r}+\frac{d(x,p)}{r^2}\right)e^{-d(x,p)^2/(4r^2)}\,[d\mu+d\nu](x,t)\\
		&\leq Cr^{1-n}\int_{\bar B_r^P(p,T)}\,[d\mu+d\nu]\\
		&\quad +Cr^{2-n}\int_{ (M\setminus B_r(p))\times[T-r^2,T]}\left(\frac{1}{r}+\frac{d(x,p)}{r^2}\right)e^{-d(x,p)^2/(4r^2)}\,[d\mu+d\nu]
	\end{align*}
	whenever $T-r^2\ge 0$, which holds for $r$ small enough.
	Since we know from \cref{ener.bd.bis} that $\mu_t(\bar B_r(p))\leq C(T)r^{n-2}$ for times $t\ge T/2$, clearly 
	$$\mu(\bar B_r^P(p,T))\leq Cr^n$$
	for $r$ small enough,
	and since $(p,T)\in G$, it follows that there is a different constant $C(p,T)$ such that
	$$\nu(\bar B_r^P(p,T))\leq Cr^n.$$
	Moreover, note that for any smooth nonnegative function $0\leq f\in C^{\infty}([0,\infty))$ with fast decay at infinity, writing $f(\lambda)=\int_\lambda^\infty(-f')(s)\,ds$ and applying Fubini's theorem we get
	\begin{align*}
		& \int_{ (M\setminus B_r(p))\times[T-r^2,T]} f(d(x,p))\,[d\mu+d\nu] \\
		& =\int_r^{\infty}(-f'(s))(\mu+\nu)( B_s(p)\times[T-r^2,T])\,ds - f(r)(\mu+\nu)( B_r(p)\times[T-r^2,T])\\
		&\leq \int_r^{\infty}|f'(s)|(\mu+\nu)( B_s(p)\times[T-r^2,T])\,ds\\
		&\leq \int_r^{\infty}|f'(s)|(\mu+\nu)(B_s^P(p,T))\,ds\\
		&\leq C(p,T)\int_r^{\infty}|f'(s)|s^n\,ds,
	\end{align*} 
	where we used the preceding bounds for $\mu(\bar{B}_r^P(p,T))$ and $\nu(\bar B_r^P(p,T))$ in the last inequality.
	In particular, applying this with $f(s):=\left(\frac1r+\frac{s}{r^2}\right)e^{-s^2/4r^2}$, and returning to the computation for $|\mu_{T-r^2}(\phi_r)-\mu_T(\phi_r)|$, we see that
	\begin{align*}
		|\mu_{T-r^2}(\phi_r)-\mu_T(\phi_r)|&\leq Cr+Cr^{2-n}\int_r^{\infty} \frac{e^{-s^2/4r^2}}{r^2}(1+s/r+(s/r)^2) s^n\,ds\\
		&\leq Cr+Cr^{-n}\int_1^{\infty}e^{-\sigma^2/4}(1+\sigma+\sigma^2)r^n\sigma^n\cdot r\,d\sigma\\
		&\leq Cr+Cr\int_1^{\infty}e^{-\sigma^2/4}(1+\sigma+\sigma^2)\sigma^n\,d\sigma\\
		&\leq Cr, 
	\end{align*}
	where we set $\sigma:=s/r$ in the second inequality, concluding the proof.
\end{proof}

\subsection{Generic tangent flows} \label{gen.tan.flows}
Let $\mathcal{T}_t^\epsilon$ be the stress-energy tensor of $(u_t^\epsilon,\nabla_t^\epsilon)$ and let
$$d\mathcal{T}:=\lim_{\epsilon\to 0}\mathcal{T}_t^\epsilon\,\dvol_g\otimes dt$$ be the (subsequential) limit, which is a measure with values into symmetric endomorphisms.
Since $|\mathcal{T}|\le C\mu$ on the spacetime, we can write
$$ d\mathcal{T}=A(x,t)\,d\mu(x,t) $$
for a bounded Borel function $A(x,t)\in\operatorname{Sym}(T_xM)$.

Let $G'\subseteq G$ be the set of points $(p,T)\in G$ at which $A$ is $\mu$-approximately continuous, with respect to the parabolic metric, and let
$$G_T':=\{p:(p,T)\in G'\}.$$
Recalling \cref{G.ae}, we have that $(p,T)\in G'$ for $\mu$-a.e.\ $(p,T)\in M\times(0,\infty)$.
Since $d\mu= d\mu_t\otimes dt$ on $M\times(0,\infty)$, we have $\mu_T(M\setminus G_T')=0$ for a.e.\ $T>0$. In what follows, we fix one such $T>0$, and prove that $\mu_T$ is $(n-2)$-rectifiable. (Rectifiability could also be obtained via the celebrated results of Preiss \cite{Preiss}, as in \cite{BethuelOrlandiSmets}, but we opt for a more direct, geometric approach, which sheds more light on the structure of the limiting measure.)

Fix now $p\in G_T'$.
Since $(p,T)\in G\subseteq\spt(\mu)$, using \cref{clearing.spt} we see that
$$ \liminf_{r\to 0}\frac{\mu_{T-c_Sr^2}(B_r(p))}{r^{n-2}}>0. $$
In view of \cref{asympHK}, which gives $K(c_Sr^2,x,p)\ge cr^{-n}$ on $B_r(p)$ (for a suitable $c(M)>0$), we deduce that
$$ \tilde\Theta_{n-2}(\mu_T,p)=\Theta^P(\mu,p,T)\ge(4\pi c_S)\liminf_{r\to 0}\frac{\mu_{T-c_Sr^2}(B_r(p))}{r^{n-2}}\ge c(M),$$
for a possibly different constant $c(M)>0$.
Also, recalling \cref{mod.theta.bis}, an integration by parts (as in the proof of \cref{eq.dens}) shows that
\begin{align}\label{mod.theta.trix}
	&\tilde\Theta_{n-2}(\mu_T,p)=\frac{1}{2(4\pi)^{(n-2)/2}}\lim_{r\to 0}\int_0^\infty \frac{\mu_T(B_{r\sigma}(p))}{(r\sigma)^{n-2}}\sigma^{n-1}e^{-\sigma^2/4}\,\frac{d\sigma}{\sigma}.
\end{align}
Using the bound $\mu_T(B_r)\le C(M,\Lambda,T)r^{n-2}$, guaranteed by \cref{ener.bd.bis}, as well as the fact that $\tilde\Theta_{n-2}(\mu_T,p)\ge c(M)$, we deduce from the previous expression that we must also have the lower bound
\begin{align}\label{lower.area.bd}&\liminf_{r\to 0}\frac{\mu_T(B_r(p))}{r^{n-2}}\ge c(M,\Lambda,T)>0.\end{align}
Using normal coordinates around $p$, we can identify
a small ball $B_\delta(p)$ with $B_\delta(0)\subset\R^n$, endowed with a metric $g$ with $g_{ij}(0)=\delta_{ij}$.

Let $\operatorname{Tan}_{n-2}(\mu_T,p)$ be the (nonempty) set of tangent measures at $p$, namely limits of the form
$$ \lim_{\ell\to\infty}s_\ell^{2-n}(\delta_{s_\ell})_*\mu_T,\quad\text{for a sequence }s_\ell\to 0, $$
where $\delta_s(x):=s^{-1}x$ is the usual dilation in the Euclidean space.
We will prove the following result.

\begin{proposition}\label{tan.flat}
	For $p\in G_T'$, the endomorphism $A(p,T)$ is the orthogonal projection onto a plane $P_0$, and moreover there is a unique tangent measure, given by
	$$\operatorname{Tan}_{n-2}(\mu_T,p)=\{\Theta^P(\mu,p,T)\cdot\mathcal{H}^{n-2}\mres P_0\}.$$
\end{proposition}

Since this holds for all $p\in G_T'$ and $\mu_T(M\setminus G_T')=0$, the following is an immediate consequence, by the well-known rectifiability criterion \cite[Theorem~11.8]{SimonGMT}.

\begin{corollary}\label{rect.coroll}
    The measure $\mu_T$ is $(n-2)$-rectifiable.
\end{corollary}

To begin the proof of \cref{tan.flat}, from now on we fix a tangent measure $\hat\mu\in\operatorname{Tan}_{n-2}(\mu_T,p)$. In order to study $\hat\mu$, we study a tangent flow whose energy at time $0$ converges to $\hat\mu$.

On the ball $B_\delta(p)=B_\delta(0)$ we can trivialize our line bundle, viewing $u_t=u_t^\epsilon$ as a $\C$-valued map and writing $\nabla_t=\nabla_t^\epsilon=d-i\alpha_t^\epsilon$.
Given $s\in(0,1)$, on the dilated ball $B_{\delta/s}(0)$, with the corresponding metric $(g^s)_{ij}:=g_{ij}(s\cdot)$,
we consider the new pair $(\tilde{u}_t^{\epsilon/s},\widetilde\nabla_t^{\epsilon/s})$ obtained from the parabolic rescaling
\begin{align}\label{resc.pair}&(\tilde{u}_t^{\epsilon/s},\tilde\alpha_t^{\epsilon/s})(x,t):=(u_t^\epsilon,s\alpha_t^\epsilon)(sx,T+s^2t),\end{align}
which solves the gradient flow equations with $\tilde\epsilon:=\epsilon/s$ in place of $\epsilon$.
Note that we have the scale invariance
$$\int\varphi(sx)\,d\tilde\mu_t^{\tilde\epsilon}(x)=s^{2-n}\int\varphi(x')\,d\mu_{T+s^2t}^\epsilon(x')$$
for any $\varphi\in C^0_c(B_{\delta})$, where $\tilde\mu_t^{\tilde\epsilon}$ is the energy density of the rescaled pair (computed in the rescaled metric).

\begin{proposition}\label{diagonal.pairs}
	Along a (not relabeled) subsequence $s_\ell\to 0$,
	there exists an assignment $s_\ell\mapsto\epsilon(s_\ell)$ such that $\epsilon(s_\ell)/s_\ell\to 0$ and the rescaled pairs
	$$ (\tilde{u}_t^{\epsilon(s_\ell)/s_\ell},\tilde\alpha_t^{\epsilon(s_\ell)/s_\ell})(x,t) $$
	satisfy the following as $\ell\to\infty$:
	their energy densities
	$$\tilde\mu_t^{\tilde\epsilon}\rightharpoonup^*\hat\mu$$
	for all $t\in\R$, as measures on $\R^n$, the shrinking measures
	\begin{align}\label{vanishing.nu.bis} &(|\dot{\tilde u}|^2+\tilde\epsilon^2|\dot{\tilde\alpha}|^2)\,dx\otimes dt\rightharpoonup^* 0, \end{align}
	and the stress-energy tensors
	$$ \widetilde{\mathcal{T}}_t^{\tilde\epsilon}\,dx\otimes dt\rightharpoonup^* A(p,T)\,d\hat\mu(x)\otimes dt, $$
	where the last two limits hold in the sense of measures on $\R^n\times \R$ (abusing notation, we let $\tilde\epsilon:=\epsilon(s_\ell)/s_\ell$, and we often drop the subscript $t$ and the superscript $\tilde\epsilon$ when no ambiguity arises).
\end{proposition}

In the sequel, we will denote by $\widetilde{\mathcal{T}}$ the $\operatorname{Sym}_n$-valued measure given by
$$ d\widetilde{\mathcal{T}}:=A(p,T)\,d\hat\mu(x)\otimes dt. $$

\begin{remark}\label{scaled.bounds}
	Some bounds from the previous sections relied on global arguments, exploiting compactness of $M$ (e.g., bounding the discrepancy $\xi^\epsilon\le C$). However, these bounds still hold for the rescaled pairs, simply by scale invariance; actually, some of them improve: for instance, the discrepancy of the rescaled pair satisfies $\tilde\xi^{\tilde\epsilon}\le Cs_\ell$.
\end{remark}

\begin{proof}[Proof of \cref{diagonal.pairs}]
	By definition of tangent measure, we have
	$$ s^{2-n}(\delta_s)_*\mu_T\rightharpoonup^*\hat\mu $$
	along a suitable sequence $s=s_\ell\to 0$.
	Since $\mu(B_s^P(p,T))\le Cs^n$ for $s$ small enough (by \cref{ener.bd.bis}), up to subsequences we have
	$$ s^{-n}(\delta_s^P)_*\mu\rightharpoonup^*\tilde\mu, $$
	for a suitable Radon measure $\tilde\mu$ on $\R\times\R^n$, where $\delta_s^P(x,t):=(s^{-1}x,s^{-2}t)$ is the parabolic dilation. By approximate continuity of the density $A=\frac{d\mathcal{T}}{d\mu}$ at $(p,T)$, we have
	$$ s^{-n}(\delta_s^P)_*\mathcal{T}\rightharpoonup^* A(p,T)\tilde\mu. $$
	Also, for fixed $s\in(0,1)$ and $k\in\{1,2,\dots\}$, note that the pair \cref{resc.pair} satisfies
	$$ (|\dot{\tilde u}|^2+\tilde\epsilon^2|\dot{\tilde\alpha}|^2)(x,t)
	=s^2(|\dot{u}|^2+\epsilon^2|\dot{\alpha}|^2)(sx,T+s^2t), $$
	where we use the metric $g^s$ in the left-hand side and $g$ in the right-hand side, and hence
	$$ \int_{B_k(0)\times [-k,k]}(|\dot{\tilde u}|^2+\tilde\epsilon^2|\dot{\tilde\alpha}|^2)\,dt\, \dvol_{g^s}
	\le s^{2-n}\nu^\epsilon(B_{ks}(p)\times[T-ks^2,T+ks^2]). $$
	In the limit $\epsilon\to 0$, assuming $s$ is so small that $ks^2<T/2$, the last measure is bounded by
	$$ \nu(B_{ks}(p)\times[T-ks^2,T+ks^2])\le \nu(\bar B^P_{ks}(p,T))\le C\mu(\bar B^P_{5ks}(p,T))\le C(ks)^n $$
	(in view of the definition of $G$ and \cref{ener.bd.bis}),
	and we deduce that
	\begin{align}\label{vanishing.nu} &\limsup_{\epsilon\to 0}\int_{B_k(0)\times[-k,k]}(|\dot{\tilde u}|^2+\tilde\epsilon^2|\dot{\tilde\alpha}|^2)\,dt\,dx\le Ck^ns^2. \end{align}
	A standard diagonal argument now gives a function $\epsilon(s)$ satisfying the first claim for $t=0$, as well as the second one, the third one with limit measure $A(p,T)\tilde\mu$,
	and the convergence of the energy densities $\tilde\mu^{\tilde\epsilon}\rightharpoonup^*\tilde\mu$ (in the spacetime) to the parabolic tangent measure $\tilde\mu$.
	
	By the semi-decreasing property, we can also assume that $\tilde\mu_t^{\tilde\epsilon}\rightharpoonup^*\tilde\mu_t$ for all $t\in\R$, and write
	$$d\tilde\mu= d\tilde\mu_t\otimes dt.$$
	Using \cref{epsBrakke.bis} for the rescaled pairs, together with Cauchy--Schwarz, we get
	$$ |\tilde\mu_0^{\tilde\epsilon}(\phi)-\tilde\mu_t^{\tilde\epsilon}(\phi)|\le C\int_{\R^n\times[0,t]}|\phi|\,d\tilde\nu^{\tilde\epsilon}
	+C\left(\int_{\R^n\times[0,t]}|d\phi|\,d\tilde\mu^{\tilde\epsilon}\right)^{1/2}\left(\int_{\R^n\times[0,t]}|d\phi|\,d\tilde\nu^{\tilde\epsilon}\right)^{1/2} $$
	(with $[0,t]$ replaced by $[t,0]$ if $t<0$), and recalling that $\tilde\nu^{\tilde\epsilon}\rightharpoonup^* 0$ by \cref{vanishing.nu.bis}, we see that in fact
	$$\tilde\mu_t=\tilde\mu_0=\hat\mu$$
	for all $t\in\R$,
	showing the first claim in full. It also follows that $A(p,T)\tilde\mu=\widetilde{\mathcal{T}}$, so that the third claim is also settled.
\end{proof}

\subsection{Proof of rectifiability}
We are left to prove \cref{tan.flat}. Let $\varphi\in C^1_c(\R)$ be a nonnegative function with $\int\varphi=1$.
Passing to the limit in \cref{div.id}, applied with the rescaled pairs, and recalling \cref{vanishing.nu.bis}, we see that
$$ 0=\lim_{r\to 0}\int_{\R^n\times\R}\ang{\varphi(t)DV(x),\widetilde{\mathcal{T}}_t^{\tilde\epsilon}}\,dx\otimes dt
=\int_{\R^n\times\R}\varphi(t)DV(x)\,d\widetilde{\mathcal{T}}(x,t) $$
(with the implicit pairing between symmetric endomorphisms in the last integral),
for all vector fields $V\in C^1_c(\R^n)$.
Since $d\widetilde{\mathcal{T}}=A(p,T)\,d\hat\mu\otimes dt$, we obtain
$$ \divergence(A(p,T)\hat\mu)=0$$
in the weak sense.

We now show that, actually, $A(p,T)\hat\mu$ defines a generalized stationary $(n-2)$-varifold (see \cite[eq.~(6.10)]{PigatiStern} for the definition used here).

\begin{proposition} \label{prop.stat.var}
	The measure $A(p,T)\hat\mu$ defines a generalized stationary $(n-2)$-varifold with weight $\hat\mu$.
\end{proposition}

\begin{proof}
	We need to show that
	$$ -I\le A(p,T)\le I,\quad\operatorname{tr}A(p,T)\ge n-2. $$
	The proof follows the same lines of \cite[Lemma~6.3]{PigatiStern}, which deals with the static situation.
	Calling $\tilde e=\tilde e^{\tilde\epsilon}$ the energy density, it amounts to showing the two bounds
	$$ \left|\int_{\R^n\times\R}\ang{\widetilde{\mathcal{T}}^{\tilde\epsilon},X\otimes X}\,dx\,dt\right|
	\le\int_{\R^n\times\R}\tilde e^{\tilde\epsilon}|X|^2\,dx\,dt, $$
	$$ \int_{\R^n\times\R}\ang{\widetilde{\mathcal{T}}^{\tilde\epsilon},\varphi I}\,dx\,dt
	\ge (n-2)\int_{\R^n\times\R}\tilde e^{\tilde\epsilon}\varphi\,dx\,dt-\delta_{\tilde\epsilon}(\varphi), $$
	with an error $\delta_{\tilde\epsilon}(\varphi)$ which is infinitesimal as $\tilde\epsilon\to 0$ (or, more precisely, as $\ell\to\infty$),
	for any compactly supported vector field $X$ and function $\varphi$ defined on the spacetime
	(each quantity is measured with respect to the rescaled metric $g^s$). Recalling that $\widetilde{\mathcal{T}}=A(p,T)\tilde\mu$, the claim follows from these two bounds once we pass to the limit.
	
	While the first one is immediate, in our setting the second one is also easier to obtain compared to \cite{PigatiStern}, since by \cref{scaled.bounds}, on any compact subset $K\subset\R\times\R^n$, eventually the discrepancy $\tilde\xi$ of the rescaled pair $(\tilde u,\widetilde\nabla)$ is bounded above by $Cs$ (with $s$ the corresponding scale). Hence, we can immediately conclude that
	$$ \int_{K} \left(\tilde\epsilon^2|\tilde\omega|^2-\frac{(1-|\tilde u|^2)^2}{4\tilde\epsilon^2}\right)
	\le C\int_{K} \sqrt{\tilde e}(\tilde\xi)^+
	\le C(K)s, $$
	since we have the uniform bound
	$$ \int_{K}\tilde e\le C(K), $$
	which easily follows from \cref{ener.bd.bis} and scale invariance. From this observation, the proof follows as in \cite{PigatiStern}.
\end{proof}

Since the limiting energy density $\tilde\mu_t=\hat\mu$ is constant in $t$, we can apply \cref{clearing.spt} to the rescaled pairs and obtain that $\liminf_{s\to 0}\frac{\hat\mu(B_r(x))}{r^{n-2}}>0$ for all $x\in\spt(\hat\mu)$.
In view of \cite[Theorem~3.8(c)]{AmbrosioSoner}, we conclude that $A(p,T)\hat\mu$ is a standard rectifiable stationary $(n-2)$-varifold, meaning that $\hat\mu$ is the weight of a stationary rectifiable varifold $V$ and the constant matrix $A(p,T)$ represents the orthogonal projection onto a plane $P_0$, which is the tangent plane at a.e.\ point. In particular, the density $\Theta_{n-2}(\hat\mu,\cdot)$ exists everywhere.

Note that we have
$$ \theta_0:=\Theta_{n-2}(\hat\mu,0)=\tilde\Theta_{n-2}(\hat\mu,0)=\tilde\Theta_{n-2}(\mu_T,p), $$
where the last equality follows from \cref{mod.theta.trix}: in fact, calling $\tilde\Theta_{n-2}^r(\hat\mu,0)$
the quantity inside the limit appearing in \cref{mod.theta.trix}, we see that
$$ \tilde\Theta_{n-2}^r(\hat\mu,0)=\lim_{\ell\to\infty}\tilde\Theta_{n-2}^{rs_\ell}(\mu_T,0)\equiv\tilde\Theta_{n-2}(\mu_T,p), $$
where the first equality comes from dominated convergence and the fact that
$$ \frac{\mu_T(B_{rs_\ell\sigma}(0))}{(rs_\ell\sigma)^{n-2}}\to\frac{\hat\mu(B_{r\sigma}(0))}{(r\sigma)^{n-2}} $$
for a.e.\ $\sigma$, which holds as $\hat\mu$ is a tangent measure. Since $\frac{\hat\mu(B_r(0))}{r^{n-2}}$ is increasing by the monotonicity formula for classical stationary varifolds, the last identity implies that it is in fact constant, i.e.,
$$ \frac{\hat\mu(B_r(0))}{\omega_{n-2}r^{n-2}}\equiv\theta_0. $$

We claim that
$$\hat\mu=\theta_0\cdot\mathcal{H}^{n-2}\mres P_0,$$
which finishes the proof.
Note that $0\in\spt(\hat\mu)$, thanks to \cref{lower.area.bd} and the fact that $\hat\mu$ is a tangent measure.
Assume that $P_0=\R^{n-2}\times\{0\}$, up to a rotation. Using coordinates $x=(y,z)\in\R^{n-2}\times\R^2$ and taking $\eta\in C^1_c(\R^2)$ with $\eta(0)=1$ and $0\le\eta\le 1$, we see that the varifold $V_\eta$, obtained from $V$ by multiplying the density pointwise by $\eta(z)$, is still stationary (since, viewing $\eta$ as a function on $\R^n$, $d\eta$ vanishes along $P_0$) and has density $\theta_0$ at the origin.
Hence, by monotonicity,
$$\theta_0\le\frac{|V_\eta|(B_r(0))}{\omega_{n-2}r^{n-2}}\le\frac{|V|(B_r(0))}{\omega_{n-2}r^{n-2}}=\theta_0$$
for all radii $r>0$. This shows that $V_\eta=V$ for all such functions $\eta$, meaning that $V$ is supported on the plane $P_0$. By the constancy theorem, $V$ has constant density $\theta_0$, so that
$$ \hat\mu=|V|=\theta_0\cdot\mathcal{H}^{n-2}\mres P_0. $$

\section{Integrality} \label{sec.int}
We establish in this section integrality of ($\frac{1}{2\pi}$ times) the density of the limiting measures $\mu_t$. 
As in the previous section, we fix a time $T>0$ such that $\mu_T(M\setminus G_T')=0$, so that by \cref{rect.coroll} $\mu_T$ is rectifiable.
We claim that $\Theta_{n - 2}(\mu_T,p)\in 2\pi\N$ for $\mu_T$-a.e.\ $p$. In order to show this, we fix a point $p\in G_T'$. Recalling \cref{tan.flat},
the usual Euclidean density $\Theta_{n-2}(\mu_T,p)$ exists; we let
$$ \theta_0:=\Theta_{n-2}(\mu_T,p)=\Theta^P(\mu,p,T). $$
Up to rotation of the normal coordinates at $p$, we assume that the tangent plane is $P_0=\R^{n-2}\times\{0\}$.

Again, in order to obtain quantization, we look at the sequence of rescaled pairs
$$ (\tilde u^{\tilde\epsilon},\tilde\alpha^{\tilde\epsilon})=(\tilde{u}_t^{\epsilon(s_\ell)/s_\ell},\tilde\alpha_t^{\epsilon(s_\ell)/s_\ell})(x,t) $$
provided by \cref{diagonal.pairs}.

\begin{proposition} As $\tilde\epsilon\to 0$ (or, more precisely, as $\ell\to\infty$) we have
	\begin{align}\label{par.van} &\sum_{j=1}^{n-2}(|\widetilde\nabla_{\de_j}\tilde u|^2+\tilde\epsilon^2|\tilde\omega(\de_j,\cdot)|^2)\,dx\otimes dt\rightharpoonup^* 0 \end{align}
	as measures in the spacetime $\R^n\times\R$.
\end{proposition}

\begin{proof}
	Identifying $P_0$ with the orthogonal projection matrix, the claim is equivalent to
	$$ \ang{\widetilde\nabla\tilde u^*\widetilde\nabla\tilde u+\tilde\epsilon^2\tilde\omega^*\tilde\omega,P_0} \,dx\otimes dt\rightharpoonup^* 0. $$
	(Note that expressions such as $\widetilde\nabla\tilde u^*\widetilde\nabla\tilde u$, as well as the pairing, in principle depend on the metric;
	however, since the rescaled metric converges smoothly to the Euclidean one, the outcome is the same up an error which is infinitesimal with respect to the energy density $\tilde e_t^{\tilde\epsilon}$.)
	Recalling the definition of the stress-energy tensor $\widetilde{\mathcal{T}}_t^{\tilde\epsilon}$, this is the same as
	$$ \ang{\tilde e_t^{\tilde\epsilon} I-\widetilde{\mathcal{T}}_t^{\tilde\epsilon},P_0} \,dx\otimes dt\rightharpoonup^* 0. $$
	Equivalently, we are claiming that
	$$ ((n-2)\tilde e_t^{\tilde\epsilon}-\ang{\widetilde{\mathcal{T}}_t^{\tilde\epsilon},P_0}) \,dx\otimes dt\rightharpoonup^* 0. $$
	However, this follows immediately from the results of the previous section, where we saw that
	$$ (n-2)\tilde e_t^{\tilde\epsilon} \,dx\otimes dt\rightharpoonup^* \theta_0 \,d\mathcal{H}^{n-2}\mres P_0(x)\otimes dt $$
	and
	\begin{align*}
		&\widetilde{\mathcal{T}}_t^{\tilde\epsilon}\,dx\otimes dt\rightharpoonup^* P_0\cdot\theta_0 \,d\mathcal{H}^{n-2}\mres P_0(x)\otimes dt. \qedhere
	\end{align*}
\end{proof}

In the sequel, for simplicity of notation, we denote by $(u_t^\epsilon,\nabla_t^\epsilon)$ the rescaled pairs, whose energy densities $e_t^\epsilon$ converge to $\theta_0\cdot\mathcal{H}^{n-2}\mres P_0$ for each time $t\in\R$, and we often drop $t$ and $\epsilon$ when no ambiguity arises.

\subsection{Slicing in time and in space}
In order to reduce to a two-dimensional setting, as in \cite[Section~6.2]{PigatiStern}, we will first select a suitable temporal slice $\R^n\times\{t_0\}=\R^n\times\{t_0^\epsilon\}$ (time $t_0$ corresponds to $T+s^2t_0$ before the rescaling), at which the gradient flow solutions resemble stationary points. Then, recalling that $P_0=\R^{n-2}\times\{0\}$, we will further slice along the first $n-2$ spatial components, and we will show that for typical $y_0=y_0^\epsilon\in\R^{n-2}$,
the energy along the two-dimensional slice $\{y_0\}\times\R^2\times\{t_0\}$ is close to $2\pi\N$, exploiting well-known energy quantization results for entire solutions on the plane.

Using coordinates $(y,z)\in\R^{n-2}\times\R^2$ for the space $\R^n$, we define the function
$$ h^\epsilon(y,t):=\int_{B_2^2(0)}e_t^\epsilon(y,z)\,dz, $$
which gives the energy on the two-dimensional slice $\{y\}\times B_2^2(0)\times\{t\}$.
We have the $L^1$ bound
$$ \int_{B_2^{n-2}(0)\times [-2,2]}h^\epsilon(y,t)\,dy\,dt=\int_{B_2^{n-2}(0)\times B_2^2(0)\times [-2,2]}e_t^\epsilon(x)\,dx\,dt\le C, $$
for a constant $C$ independent of $\epsilon$
(in fact, the integral converges to $2^n\omega_{n-2}\theta_0$, by assumption).
Using the parabolic distance on $\R^{n-2}\times\R$, we define the parabolic maximal function
$$M^Ph^{\epsilon}(y,t):=\sup_{r\in(0,1)}r^{-n}\int_{B^P_r(y,t)}h^{\epsilon}.$$
Since the Lebesgue measure on $\R^{n-2}\times \R$ is doubling with respect to this distance, we have the usual weak-(1,1) estimate on the operator $M^P$. In view of the previous $L^1$ bound, we get in particular
\begin{align}\label{K.bd} &|\{(y,t)\in B_1^{n-2}(0)\times[-1,1] : M^P h^\epsilon(t,y)>K\}|\le \frac{C}{K} \end{align}
for all $K>0$, with the volume measured using the Lebesgue measure $\mathcal{L}^{n-1}$.

In the sequel, we often omit the center of a Euclidean ball, when centered at the origin.

\begin{proposition}
	For a fixed $K$ large enough, the following holds. Eventually there exists $t_0=t_0^\epsilon\in[-1,1]$, depending on $\epsilon$, such that
	\begin{align}\label{vanishing.ann} &\int_{B_2^{n-2}\times [B_2^2\setminus B_1^2]}e_{t_0}^\epsilon(y,z)\,dy\,dz\to 0, \end{align}
	\begin{align}\label{vanishing.tris} &\int_{ B_2^{n-2}\times B_2^2\times\{t_0\}}(|\dot{u}|^2+\epsilon^2|\dot{\alpha}|^2)\,dy\,dz\to 0, \end{align}
	\begin{align}\label{par.van.bis} &\int_{ B_2^{n-2}\times B_2^2\times\{t_0\}}\sum_{j=1}^{n-2}(|\nabla_{\de_j}u|^2+\epsilon^2|\omega(\de_j,\cdot)|^2)\,dy\,dz\to 0, \end{align}
	as well as
	$$ |\{y\in B_1^{n-2}:M^Ph^\epsilon(y,t_0)>K\}|\le\frac{1}{2}|B_1^{n-2}|. $$
\end{proposition}

\begin{proof}
	Using the fact that
	$$ \int_{ B_2^{n-2}\times [B_2^2\setminus B_1^2]\times[-1,1]}e_{t}^\epsilon(y,z)\,dy\,dz\,dt\to 0, $$
	as well as \cref{vanishing.nu.bis} and \cref{par.van}, we can find a sequence $\delta_\epsilon\to 0$
	such that the quantities appearing in the first three claims can exceed $\delta_\epsilon$ only for times $t_0$ in a set $F\subseteq[-1,1]$ (depending on $\epsilon$) of measure $|F|\le\delta_\epsilon\to0$.
	
	Moreover, calling $F'\subseteq[-1,1]$ the set of times where the last claim fails, we have
	$$ |F'| \cdot \frac{1}{2}|B_1^{n-2}|\le \frac{C}{K}, $$
	thanks to the previous volume bound \cref{K.bd}. Hence, we have $|F\cup F'|\le\delta_\epsilon+\frac{C}{K}$,
	so that for $K$ large enough (and $\epsilon$ small enough) the complement $[-1,1]\setminus(F\cup F')$ is nonempty.
	It then suffices to take $t_0\in[-1,1]\setminus(F\cup F')$.
\end{proof}

In the spacetime, the energy density of the shifted pairs $(u_{t_0^\epsilon+t}^\epsilon,\alpha_{t_0^\epsilon+t}^\epsilon)$ still converges to $\theta_0\cdot\mathcal{H}^{n-2}\mres P_0$, for all $t\in\R$. Indeed, calling $\mu_t'$ the limit, which we can assume to exist for all $t$ (up to a subsequence, thanks to the semi-decreasing property), we know that
$$ \theta_0 \,d\mathcal{H}^{n-2}\mres P_0(x)\otimes dt= d\mu_t'(x)\otimes dt, $$
since in the spacetime the energy density still converges to $\theta_0\,\mathcal{H}^{n-1}\mres ( P_0\times\R)$.
Hence, we must have $\mu_t'=\mathcal{H}^{n-2}\mres P_0$ for a.e.\ $t$, and by the semi-decreasing property we conclude that this must hold for all $t\in\R$.

From now on, we will replace $(u_t^\epsilon,\alpha_t^\epsilon)$ with the pair shifted by time $t_0^\epsilon$; in other words, we will assume without loss of generality that $t_0^\epsilon=0$, while retaining the property that
$$ \mu_t^\epsilon\rightharpoonup^*\theta_0\cdot\mathcal{H}^{n-2}\mres P_0, $$
as well as the conclusions of the previous proposition. In particular, we have
$$ \int_{B_2^{n-2}\times B_2^2}e_0^\epsilon(y,z)\,dy\,dz\le C $$
for a constant $C$ independent of $\epsilon$.
A useful consequence of the last bound, together with \cref{vanishing.tris} and \cref{div.id}, is that
\begin{align}\label{vanishing.div} &\int_{B_2^{n-2}\times B_2^2}\lvert\divergence\mathcal{T}_0^\epsilon\rvert(y,z)\,dy\,dz\to 0. \end{align}

Finally, we also slice in the first $n-2$ spatial coordinates.

\begin{proposition}
	We have
	$$ \int_{B_1^{n-2}}\left|\int_{B_2^2}e_0^\epsilon(y,z)\,dz-\theta_0\right|\,dy\to 0 $$
	as $\epsilon\to 0$.
\end{proposition}

\begin{proof}
	Indeed, let $h^\epsilon(y):=\int_{B_2^2}e_0^\epsilon(y,z)\,dz$, whose average on $B_1^{n-2}$ converges to $\theta_0$ (as $\mu_0^\epsilon\rightharpoonup^*\theta_0\cdot\mathcal{H}^{n-2}\mres P_0$).
	Fix a cut-off function $\chi\in C^1_c(B_2^{2})$ with $\chi=1$ on $B_1^{2}$ and consider a vector field $Y\in C^1_c(B_2^{n-2})$.
	We claim that
	$$ \left|\int_{B_2^{n-2}}h^\epsilon\divergence(Y)\right|\le\delta_\epsilon\|DY\|_{L^\infty} $$
	for a sequence $\delta_\epsilon\to 0$. Once this is done, the statement follows from Allard's strong constancy lemma \cite[Theorem~1.(4)]{Allard}.
	In order to show the claim, note that by \cref{vanishing.ann} we have
	$$ \int_{B_2^{n-2}}h^\epsilon\divergence(Y)=\int_{\R^n}h^\epsilon\divergence(\chi(z)Y(y))\,dy\,dz+o(\|DY\|_{L^\infty}). $$
	The latter integral equals
	$$ \int_{\R^n}\ang{e_0^\epsilon I,D(\chi(z) Y(y))}=\int_{\R^n}\ang{\mathcal{T}_0^\epsilon+(\nabla u^*\nabla u+\epsilon^2\omega^*\omega),D(\chi(z)Y(y))}\,dy\,dz, $$
	again up to errors (due to the change of metric) that vanish as $\epsilon\to 0$. The term $\nabla u^*\nabla u+\epsilon^2\omega^*\omega$ has a vanishing contribution on the region where $\chi(z)=1$, thanks to \cref{par.van.bis} and the fact that $Y$ is parallel to $P_0$, and also on the complement, thanks to \cref{vanishing.ann}.
	Hence,
	$$ \int_{B_2^{n-2}}h^\epsilon\divergence(Y)=-\int_{\R^n}\ang{\divergence(\mathcal{T}_0^\epsilon),\chi(z)Y(y)}\,dy\,dz+o(\|DY\|_{L^\infty}), $$
	and the conclusion follows from \cref{vanishing.div}.
\end{proof}

\begin{proposition}
	Eventually we can find $y_0=y_0^{\epsilon}\in B_1^{n-2}$ such that the following conclusions hold:
	$$ \sup_{r\in(0,1)}r^{2-n}\int_{B_r^{n-2}(y_0)\times [B_2^2\setminus B_1^2]}e_{0}^\epsilon(y,z)\,dy\,dz\to 0, $$
	\begin{align}\label{vanishing.quadris} &\sup_{r\in(0,1)}r^{2-n}\int_{ B_r^{n-2}(y_0)\times B_2^2\times\{0\}}(|\dot{u}|^2+\epsilon^2|\dot{\alpha}|^2)\,dy\,dz\to 0, \end{align}
	\begin{align}\label{near.flat.aux} &\sup_{r\in(0,1)}r^{2-n}\int_{ B_r^{n-2}(y_0)\times B_2^2\times\{0\}}\sum_{j=1}^{n-2}(|\nabla_{\de_j}u|^2+\epsilon^2|\omega(\de_j,\cdot)|^2)\,dy\,dz\to 0, \end{align}
	$$ \sup_{r\in(0,1)}r^{2-n}\int_{B_r^{n-2}(y_0)\times B_2^2}\lvert\divergence\mathcal{T}_0^\epsilon\rvert(y,z)\,dy\,dz\to 0, $$
	as well as
	\begin{align}\label{allard.max} &\sup_{r\in(0,1)}r^{2-n}\left|\int_{B_r^{n-2}(y_0)\times B_2^2}e_0^\epsilon(y,z)\,dy\,dz-\theta_0|B_r^{n-2}|\right|\to 0 \end{align}
	and, for some $C$ independent of $\epsilon$,
	\begin{align}\label{par.max.en} &\sup_{r\in(0,1)}r^{-n}\int_{ B_r^{n-2}(y_0)\times B_2^2\times[-r^2,r^2]}e_t^\epsilon(y,z)\,dt\,dy\,dz\le C. \end{align}
\end{proposition}

\begin{proof}
	In view of the previous proposition, we have $M^Ph^\epsilon(y_0,0)\le K$, and hence the last conclusion (with $C:=K$), for all $y_0\in B_1^{n-2}$ outside of a set of measure at most $\frac{1}{2}|B_1^{n-2}|$.
	
	The other bounds follow from the previous propositions, together with the following elementary fact: if we have $\int_{B_2^{n-2}}f_k\to 0$ for a family of nonnegative functions $f_k$, then there exists an infinitesimal sequence $\delta_k\to 0$ such that
	$$ \left|\left\{y\in B_1^{n-2}:\sup_{r\in(0,1)}\int_{B_r^{n-2}(y)}f_k>\delta_k\right\}\right|\le C(n)\delta_k. $$
	The latter fact follows by taking $\delta_k:=(\int_{B_2^{n-2}}f_k)^{1/2}$ and applying the weak-(1,1) bound for the maximal operator on the Euclidean space $\R^{n-2}$.
\end{proof}

\subsection{Analysis of the two-dimensional slice}
For each $s\in (\epsilon,1)$, consider the set
$$\mathcal{V}^{\epsilon}(s)\subseteq B_2^2(0)$$
consisting of those points $z\in B_2^2(0)$ for which
$$r(y_0,z,0)<s,$$
where we define $r(p,t)$ as in \cref{decay} by
$$r(p,t):=\sup\{s\in(0,1):|u|^2\geq 1-\beta\text{ on } \bar{B}_s(p)\times[t-s^2,t]\},$$
with $\beta=\beta_D$ the constant appearing there, and $r(p,t):=0$ if the last set is empty.

\begin{proposition}\label{conv.to.1}
	We have $|u^\epsilon|\to 1$, locally uniformly on $(\R^n\setminus P_0)\times\R$.
\end{proposition}

\begin{proof}
	This is an immediate consequence of \cref{clearing.values}, applied to the rescaled pairs (recall also \cref{scaled.bounds}).
\end{proof}

\begin{lemma}\label{cover.bd}
	There exist constants $C$ and $N_0$, independent of $\epsilon$ and $s$, such that $\mathcal{V}^{\epsilon}(s)$ is covered by a collection of $N\leq N_0$ disks $B_{Cs}^2(z_1),\dots,B_{Cs}^2(z_N)$ of radius $Cs$. 
\end{lemma}

\begin{proof}
	We can assume without loss of generality that $s<\frac{1}{10}$.
	By definition, if $z\in \mathcal{V}_{\epsilon}(s)$, then $r(y_0,z,0)\leq s$, so there exists an antecedent point $(x',t')\in  \bar{B}_s(y_0,z)\times[-s^2,0]$ for which $|u^{\epsilon}|^2(x',t')< 1-\beta$. It then follows from \cref{clearing.values} that
	$$\mu_{\tau}^{\epsilon}(B_{4s}(x'))\ge cs^{n-2}$$
	for all $\tau\in [t'-4c's^2,t'-c's^2]$, for suitable constants $c,c'>0$.
	Also, denoting $x'=(y',z')$, by \cref{conv.to.1} we must have $|z'|<\frac{1}{10}$ eventually, and thus
	$$ B_{4s}^2(z')\subseteq B_2^2(0). $$
	Consequently, setting $\mathcal{C}:=\R^{n-2}\times B_2^2(0)\times\R$, we have
	\begin{equation}\label{mu.l.bd}
		(\mu^{\epsilon}\mres\mathcal{C})(B_{c_0s}(y_0,z)\times(-c_0s^2,c_0s^2) )\ge cc' s^n
	\end{equation}
	for some constant $c_0\ge 1$ and all $z\in \mathcal{V}_{\epsilon}(s)$. Next, note that the parabolic balls
	$$B^P_{c_0s}(y_0,z,0)\supseteq B_{s}^{n-2}(y_0)\times\{z\}\times(-s^2,s^2)$$
	cover 
	$$S:= B_s^{n-2}(y_0)\times \mathcal{V}^{\epsilon}(s)\times(-s^2,s^2),$$
	so we can apply Vitali's covering lemma (with respect to the parabolic metric) to deduce the existence of points
	$$z_1,\ldots,z_N\in \mathcal{V}^{\epsilon}(s)$$
	for which the balls $B^P_{c_0s}(y_0,z_j,0)$ are disjoint while $B^P_{5c_0s}(y_0,z_j,0)$ cover $S$. By the disjointness of the balls and \cref{mu.l.bd}, we have
	\begin{align*}
		&Ncc' s^n\leq \sum_{j=1}^N\mu^{\epsilon}\mres\mathcal{C}(B^P_{c_0s}(y_0,z_j,0))
		\leq \mu^{\epsilon}(B_{c_0s}(y_0)\times B_{2}^2(0)\times (-c_0^2s^2,c_0^2s^2)).
	\end{align*}
	Using \cref{par.max.en} (or \cref{ener.bd.bis} if $c_0s\ge 1$), it follows that
	$$Ncc's^n\leq C(c_0)s^n.$$
	In particular, $N\leq N_0$ is bounded independent of $\epsilon$ and $s$.
	Since the balls $B_{5c_0s}(z_j)$ cover the set $\mathcal{V}^\epsilon(s)$, we reach the desired conclusion.
\end{proof}

Choosing $s:=R\epsilon$ with $R$ large, after a possible reduction to make these disks disjoint and far apart, we will show that each carries an amount of energy close to $2\pi\N$.
Before doing that, we check that their complement gives a negligible contribution. For technical reasons, we look at an $\epsilon$-fattening of the slice $\{y_0\}\times\R^2$.

\begin{proposition}\label{e.decay.aux}
	We have
	$$ \lim_{R\to\infty}\lim_{\epsilon\to0}
	\epsilon^{2-n}\int_{B_\epsilon^{n-2}(y_0)\times[B_2^2(0)\setminus\mathcal{V}^\epsilon(R\epsilon)]}e_0^\epsilon(y,z)\,dy\,dz=0. $$.
\end{proposition}

\begin{proof}
	It is easy to see that the function $r(\cdot,0)$ is $1$-Lipschitz on $B_2^{n-2}(0)\times B_2^2(0)$, so that 
	$$r(y,z,0)\geq r(y_0,z,0)-\epsilon\ge \frac{r(y_0,z,0)}{2}$$
	for $y\in B_{\epsilon}^{n-2}(y_0)$ and $z\nin\mathcal{V}^\epsilon(2\epsilon)$. By \cref{decay}, we have
	\begin{equation}\label{exp.dec}
		e_t^{\epsilon}(u_{\epsilon},\nabla_{\epsilon})(p)\leq \frac{C}{\epsilon^2}e^{-2a r(p,t)/\epsilon}+C\epsilon^2
	\end{equation}
	for some positive constants $C$ and $a$.
	Thus, for any $R\ge 2$, we see that
	\begin{align*}
		\int_{B_{\epsilon}^{n-2}(y_0)\times [B_2^2(0)\setminus\mathcal{V}^{\epsilon}(R\epsilon)]}
		e_0^{\epsilon}(u_{\epsilon},\nabla_{\epsilon})
		&\leq C\int_{B_{\epsilon}^{n-2}(y_0)\times [B_2^2(0)\setminus\mathcal{V}^{\epsilon}(R\epsilon)]}
		\epsilon^{-2}e^{-2a r(y,z,0)/\epsilon}\,dy\,dz+C\epsilon^n\\
		&\leq C\epsilon^{n-2}\int_{B_2^2(0)\setminus \mathcal{V}^{\epsilon}(R\epsilon)}\epsilon^{-2}e^{-ar(z)/\epsilon}\,dz+C\epsilon^n,
	\end{align*}
	where we abbreviate $r(z):=r(y_0,z,0)$.
	We now bound the last integral using the following fact: setting $\Omega:=B_2^2(0)\setminus\mathcal{V}^\epsilon(R\epsilon)$, for any smooth function $0\le f\in C^\infty([0,\infty))$ with fast decay at infinity,
	$$ \int_{\Omega}f(r(z))\,dz=\int_{R\epsilon}^\infty (-f'(s))|\Omega\cap\mathcal{V}^\epsilon(s)|\,ds; $$
	this easily follows by writing $f(\lambda)=\int_\lambda^\infty(-f')(s)\,ds$ and using Fubini's theorem.
	Applying this in our case, we get
	\begin{align*}
		\int_{\Omega}\epsilon^{-2}e^{-ar(z)/\epsilon}\,dz
		&\leq \int_{R\epsilon}^\infty\frac{ae^{-as/\epsilon}}{\epsilon^3}|\Omega\cap\mathcal{V}^{\epsilon}(s)|\,ds \\
		&\leq \int_{R\epsilon}^\infty\frac{ae^{-as/\epsilon}}{\epsilon^3}\cdot Cs^2\,ds \\
		&=Ca\int_R^\infty e^{-a\sigma}\sigma^2\,d\sigma,
	\end{align*}
	where we used \cref{cover.bd} to bound $|\Omega\cap\mathcal{V}^{\epsilon}(s)|\le Cs^2$.
	Since the last integral is infinitesimal as $R\to\infty$, the claim follows.
\end{proof}

Next, we fix a large $R\ge 2$, and apply \cref{cover.bd} with $s=R\epsilon$ to get points
$z_1^{\epsilon},\dots,z_N^{\epsilon}\in B_2^2(0)$ such that the disks $B_{CR\epsilon}^2(z_j^{\epsilon})$ cover $\mathcal{V}^{\epsilon}(R\epsilon)$. Passing to a subsequence as $\epsilon\to 0$, we may assume moreover that $N$ is independent of $\epsilon$ and that the (possibly infinite) limits
$$r_{ij}:=\lim_{\epsilon\to 0}\frac{|z_i^{\epsilon}-z_j^{\epsilon}|}{\epsilon}$$
exist. We can now find a subcollection of indices, as well as a constant $C'\ge CR$ (depending on $R$, but not on $\epsilon$), such that the dilated disks $B_{C'\epsilon}^2(z_i^\epsilon)$ are disjoint, and their union includes all the original disks: to get this,
we introduce an equivalence relation on indices, such that the equivalence class of $i$ is $[i]=\{j:r_{ij}<\infty\}$, and we consider a subcollection given by choosing one representative for each class. Clearly, taking
$$ C':=CR+\max\{r_{ij}\mid r_{ij}<\infty\}+1, $$
the previous claim holds for $\epsilon$ small enough. Up to rearranging indices, we assume that the subcollection is $\{1,\dots,N'\}$. In summary, for $\epsilon$ small enough, we have
\begin{align}\label{disj.union}
	&\mathcal{V}^\epsilon(R\epsilon)\subseteq\bigsqcup_{i=1}^{N'}B_{C'\epsilon}^2(z_i^\epsilon)\subseteq B_2^2(0)
\end{align}
(where the union is disjoint), and
\begin{align}\label{dist.to.infty}
	&\frac{|z_i^\epsilon-z_j^\epsilon|}{\epsilon}\to\infty\quad\text{for }1\le i<j\le N'.
\end{align}

Recall that $e_t^\epsilon\le\frac{S}{\epsilon^2}$, for a suitable constant $S\ge C'$, by \cref{prop.nabla.u.bd} (recall \cref{scaled.bounds}).
We then see that, for this $S$ and for any $\delta\in (0,\frac{1}{2})$, the rescaled pair
$$ (\tilde u,\tilde\alpha)(y,z):=(u^\epsilon,\epsilon\alpha^\epsilon)(y_0+\epsilon y,z_i^\epsilon+\epsilon z), $$
for $i=1,\dots,N'$, $R$ large enough, and $\epsilon$ small enough, satisfies all the assumptions of \cref{blowup.lim} below, which is a variant of \cite[Proposition~6.7]{PigatiStern}, where we weaken the restriction that the pairs solve the $U(1)$-Higgs equations to the requirement that they nearly solve it in an $L^2$ sense.

In more detail,
note that $B_{\kappa^{-1}\epsilon}^2(z_i^\epsilon)\setminus \bar B_{S\epsilon}^2(z_i^\epsilon)$ is eventually disjoint from $\mathcal{V}^\epsilon(R\epsilon)$, owing to \cref{dist.to.infty} and the inclusion $B_{S\epsilon}^2(z_i^\epsilon)\supseteq B_{C'\epsilon}^2(z_i^\epsilon)$.
Hence, \cref{vort.contr} holds since eventually, in terms of our pairs, $|u|^2(y,z)\ge1-\beta$
for $y\in B_\epsilon^{n-2}(y_0)$ and $z\in B_{\kappa^{-1}\epsilon}^2(z_i^\epsilon)\setminus \bar B_{S\epsilon}^2(z_i^\epsilon)$
(as $r(y,z,0)\ge r(z)-\epsilon\ge R\epsilon-\epsilon>0$ here).
Also, \cref{near.flat} holds by \cref{near.flat.aux}. Moreover, \cref{e.decay} holds thanks to \cref{e.decay.aux}. Finally, \cref{near.sol} holds by \cref{vanishing.quadris} and the gradient flow equations.

\begin{lemma}\label{blowup.lim} Given $\delta\in(0,\frac{1}{2})$ and $S>1$, there exist $K(\delta,S)>S$ and $0<\kappa(\delta,S,n)<K^{-1}$ such that the following holds. Assume $(u,\nabla)$ is a smooth pair with $|u|\le1$ on the trivial line bundle over a cylinder $Q=B_{\kappa^{-1}}^{n-2}\times B^2_{\kappa^{-1}}$, endowed with a metric $g$, satisfying
	$$ \|g-\delta\|_{C^2}\le\kappa, $$
	$$e^1(u,\nabla)\leq S,$$
	\begin{equation}\label{vort.contr}
		\{|u|^2<1-\beta\}\cap (B_1^{n-2}\times B_{\kappa^{-1}}^2)\subseteq B_1^{n-2}\times \bar{B}_S^2,
	\end{equation}
	\begin{equation}\label{near.flat}
		\int_Q\sum_{j=1}^{n-2}(|\nabla_{\de_j} u|^2+|\omega(\de_j,\cdot)|^2)\leq \kappa,
	\end{equation}
	\begin{equation}\label{e.decay}
		\int_{B_1^{n-2}\times [B_{\kappa^{-1}}^2\setminus B_S^2]}e^1(u,\nabla)\leq \omega_{n-2}\delta,
	\end{equation}
	and
	\begin{equation}\label{near.sol}
		\int_Q\left|\nabla^*\nabla u-\frac{(1-|u|^2)u}{2}\right|^2+|d^*\omega-\langle iu,\nabla u\rangle|^2\le\kappa.
	\end{equation}
	Then
	$$\left|\int_{B_1^{n-2}\times B_K^2}e^1(u,\nabla)-2\pi\omega_{n-2} |D|\right|<2\omega_{n-2}\delta,$$
	where $D$ is the degree of the $S^1$-valued map $u/|u|$ on $\{0\}\times \partial B_S^2$.
\end{lemma}

\begin{proof}
	We can almost repeat the proof of \cite[Proposition~6.7]{PigatiStern} verbatim, with a few notable changes.
	Arguing by contradiction, we can again consider pairs $(u_j,\nabla_j=d-i\alpha_j)$
	defined on larger and larger domains.
	
	In contrast to the quoted proof, in our situation these pairs are only approximate solutions in the sense of \cref{near.sol}; however, in a local Coulomb gauge on a smooth domain, the pair $(u_j,\alpha_j)$ is still precompact in $W^{1,2}_{loc}$.
	Indeed, since $|\ang{iu,\nabla u}|\le\sqrt{S}$, \cref{near.sol}
	gives a local $L^2$ bound on $\Delta_H\alpha_j=d^*d\alpha_j$, and hence a local $W^{2,2}$ bound for $\alpha_j$, meaning that $\alpha_j$ is precompact in $W^{1,2}_{loc}$. Since $|(d-i\alpha_j)u_j|\le\sqrt{S}$ and $|u_j|\le 1$, this in particular gives a local $W^{1,2}$ bound for $u_j$.
	Finally, noting that $\Delta u=-\nabla^*\nabla u_j+2\ang{i\alpha_j,du_j}+|\alpha_j|^2u_j$ is locally bounded in $L^p$ for some $p>1$ (by Sobolev embedding), we obtain precompactness of $u$ in $W^{1,p}_{loc}$.
	Using also the previous pointwise bound, we then see that $(d-i\alpha_j)u_j$ is precompact in $L^2_{loc}$, which gives precompactness of $u$ in $W^{1,2}_{loc}$.
	
	Hence, as in \cite[Proposition~6.7]{PigatiStern}, after a change of gauge (obtained by interpolating between a local Coulomb gauge and the gauge such that $u_j=|u_j|e^{iD\theta}$), we still have precompactness in $W^{1,2}_{loc}$ on (domains converging to) $B_1^{n-2}\times\R^2$, obtaining a strong $W^{1,2}$ limit $(u_\infty,d-i\alpha_\infty)$ here. This limit is a critical pair, and hence smooth (see the appendix of \cite{PigatiStern}).
	
	Thanks to \cref{near.flat}, this limit pair is a $P_0$-invariant solution on $B_1^{n-2}\times\R^2$, up to a further change of gauge (see the proof of \cite[Proposition~6.7]{PigatiStern} for the details). Moreover, it follows from \cref{e.decay} that this limiting solution has finite energy, and its energy on each two-dimensional slice $\{y\}\times\R^2$ differs from the energy on $\{y\}\times B_K^2$ by at most $\delta$, regardless of the choice of $K>S$.
	
	The rest of the proof goes through essentially unchanged. In the conclusion, we consider the energy on $B_1^{n-2}\times B_K^2$ rather than the two-dimensional slice $\{0\}\times B_K^2$, since we can only invoke $W^{1,2}$ convergence of the pairs $(u_j,\alpha_j)$, as opposed to the full $C^1$ convergence.
\end{proof}

In particular, going back to our pairs, for any fixed $\delta>0$, for $R$ sufficiently large and $\epsilon$ sufficiently small, it follows that
$$\dist\left(\frac{1}{|B_{\epsilon}^{n-2}|}\int_{B_{\epsilon}^{n-2}(y_0)\times B_{K\epsilon}^2(z_i^{\epsilon})}e^{\epsilon}_0,2\pi \mathbb{N}\right)<2\delta,$$
and consequently, summing over $i=1,\ldots,N'$ and recalling \cref{e.decay.aux}, as well as \cref{disj.union}, \cref{dist.to.infty}, and the fact that $K\ge S\ge C'$, we deduce that
$$\dist\left(\frac{1}{|B_{\epsilon}^{n-2}|}\int_{B_{\epsilon}^{n-2}(y_0)\times B_2^2(0)}e^{\epsilon}_0,2\pi \mathbb{N}\right)<4N\delta.$$
In particular, recalling \cref{allard.max} and letting $\delta\to 0$, we obtain $\theta_0\in 2\pi\N$.

\section{Brakke's inequality and proof of \cref{thmBrakke}}\label{brakke.sec}
To conclude the proof of \cref{thmBrakke} we need to prove that the limiting measures $\mu_t$ satisfy Brakke's inequality \cref{brakke.ineq.bis}. This last result can be seen by analyzing separately the shrinking and transport terms appearing in the first variation \cref{epsBrakke} of the measure $\mu^{\epsilon}_{t}$. Intuitively, Brakke's inequality will follow if we can establish the following two inequalities: 
\begin{equation*}
	\int_M \langle d \phi, (T_x^\perp\mu_t)  \vec{H}_t \rangle \, d\mu_t \geq \limsup_{\epsilon \to 0} \int_M \langle \divergence(\mathcal{T}^\epsilon_t), d \phi \rangle
\end{equation*}
and 
\begin{equation*}
	\int_M \phi |\vec H_t|^2 \, d\mu_t \leq \liminf_{\epsilon \to 0} \int_M 2 \phi \left( \vert \dot{u}_{t}^\epsilon \vert^2 + \epsilon^2 \vert \dot{\alpha}_{t}^\epsilon \vert^2 \right), 
\end{equation*}
for any $\phi \in C^2(M, \mathbb{R}_{\ge 0})$. In other words, the transport and shrinking terms of our  $\epsilon$-variation should be controlled by the corresponding terms of Brakke's inequality, at least in an integral sense.

As in \cref{gen.tan.flows}, from the stress-energy tensors $\mathcal{T}_{t}^{\epsilon}$ of $(u_t^{\epsilon}, \nabla_t^{\epsilon})$, we can extract a subsequential weak-* limit 
$$d\mathcal{T}=\lim_{\epsilon\to 0}\mathcal{T}_t^{\epsilon}\,\dvol\otimes dt$$
as measures taking values in the pullback bundle of $\operatorname{Sym}(TM)$ over $M\times [0,\infty)$.
By virtue of the bound $|\mathcal{T}|\le C\mu$ and the fact that $\mu=\mu_t\otimes\mathcal{L}^1(t)$, we can write
$$d\mathcal{T} = d\mathcal{T}_t \otimes dt= A(x,t)\, d\mu_t(x)\otimes dt.$$
Moreover, it follows from the results of \cref{tf.sec} that $A(x,t)$ coincides with the unique tangent plane to $\mu_t$ at $\mu$-a.e.\ point $(x,t)\in M\times (0,\infty)$, so that the measures $\mathcal{T}_t$, which are defined for a.e.\ $t$, can be identified with a family of rectifiable $(n-2)$-varifolds, whose first variation along a vector field $X\in C^1(M)$ is given by
$$\delta \mathcal{T}_t(X)=\langle DX,\mathcal{T}_t\rangle. $$

Abusing notation slightly, for $\epsilon>0$ we write
$$\delta \mathcal{T}_t^{\epsilon}(X):=\int_M\langle DX,\mathcal{T}_t^{\epsilon}\rangle\,\dvol_g=-\int_M\langle X, \divergence (\mathcal{T}_t^{\epsilon})\rangle\, \dvol_g$$
and recast \cref{epsBrakke} as
\begin{equation}
	\frac{d}{dt} \int_{M} \phi e^\epsilon(u_t^\epsilon, \nabla_t^\epsilon) = \int_M - 2 \phi \left( \vert \dot{u}_t^\epsilon \vert^2 + \epsilon^2 \vert \dot{\alpha}_t^\epsilon \vert^2 \right) - \delta \mathcal{T}_{t}^{\epsilon}(d \phi).
\end{equation} 
In other words the transport term in \cref{epsBrakke} can be rewritten as the first variation of the stress-energy tensor $\mathcal{T}_{t}^{\epsilon}$. Reasoning analogously to \cite{Ilmanen} and \cite{AmbrosioSoner}, we expect it to converge to the mean curvature vector in a weak sense, while we expect the shrinking term to bound the corresponding term in Brakke's inequality. To prove this, introduce the vector-valued measures 
\begin{equation*}
	d\sigma^\epsilon = d\sigma_t^\epsilon\otimes dt := \divergence(\mathcal{T}^\epsilon_t) \,\dvol_g\otimes dt. 
\end{equation*}
By Cauchy--Schwarz and Young's inequality, together with \cref{div.id}, these measures are uniformly bounded in $(C^0)^*$ on $M \times [0, T]$, for any $T > 0$. Thus, after possibly passing to a subsequence, we can assume that $\sigma^\epsilon \rightharpoonup^* \sigma$ on $M \times [0, \infty)$ for a vector-valued Radon measure $\sigma$. 
Moreover, invoking the lower semicontinuity property from \cite[Remark~2.2]{AmbrosioSoner}, we have the inequalities
\begin{align*}
	\int_{M\times (0,\infty)} \left\vert \frac{d\sigma}{d\mu} \right\vert^2 \, d\mu 
	& \leq \liminf_{\epsilon\to 0} \int_{M\times (0,\infty)} \left\vert \frac{d\sigma^{\epsilon}}{d\mu^{\epsilon}} \right\vert^2 \, d\mu^{\epsilon} \\
	& \leq \liminf_{\epsilon\to 0} \int_{M\times (0,\infty)} \frac{\lvert \divergence(\mathcal{T}_{t}^{\epsilon}) \rvert^2}{e^{\epsilon}(u_t^\epsilon, \nabla_t^\epsilon)} \, \dvol_g\, dt \\
	& \leq \liminf_{\epsilon\to 0} \int_{M\times (0,\infty)} 4 \left( \vert \dot{u}_{t}^{\epsilon} \vert^2 + \epsilon^2 \vert \dot{\alpha}_{t}^{\epsilon} \vert^2 \right) \, \dvol_g \,dt \\
	& \leq 2\liminf_{\epsilon\to 0} E_{\epsilon}(u_0^\epsilon, \nabla_0^\epsilon) \\
	& \leq 2\Lambda, 
\end{align*} 
where we used \cref{div.id} and Cauchy--Schwarz in the third inequality, and \cref{enerID} in the last one. It also follows from \cite[Remark~2.2]{AmbrosioSoner} that $\sigma$ is absolutely continuous with respect to $\mu$. Thus, on $M\times(0,\infty)$ we can write
$$ d\sigma=\vec H_t\,d\mu_t\otimes dt $$
for a suitable vector-valued density $\vec H_t$, which satisfies
\begin{align}\label{br.sum.chain}
	&\int_{0}^\infty \left( \int_M \vert \vec{H}_t \vert^2 \, d\mu_t \right) \,dt\le2\Lambda.
\end{align}
A local version of this bound follows similarly, giving
\begin{equation*}
	\int_s^t\int_{M} \phi \vert \vec{H}_\tau \vert^2 \, d\mu_\tau \, d\tau \leq \liminf_{\epsilon\to0} \int_s^t\int_{M} 4 \phi \left( \vert \dot{u}_{\tau}^{\epsilon} \vert^2 + \epsilon^2 \vert \dot{\alpha}_{\tau}^{\epsilon} \vert^2 \right) \, \dvol_g \,d\tau, 
\end{equation*} 
for $\phi \in C^1(M,\R_{\ge0})$ and $0<s<t$.
We can easily make the last bound into a sharp one, replacing the factor $4$ by $2$.

\begin{proposition}\label{sharper.brakke}
	We have $$\int_s^t\int_{M} \phi \vert \vec{H}_\tau \vert^2 \, d\mu_\tau \, d\tau \leq \liminf_{\epsilon\to0} \int_s^t\int_{M} 2 \phi \left( \vert \dot{u}_{\tau}^{\epsilon} \vert^2 + \epsilon^2 \vert \dot{\alpha}_{\tau}^{\epsilon} \vert^2 \right) \, \dvol_g \,d\tau.$$
\end{proposition}

\begin{proof}
	In fact, using Cauchy--Schwarz as above, \cref{div.id} gives
	$$ \lvert\divergence(\mathcal{T}_{t}^{\epsilon})\rvert^2\le 4(|\nabla_t^\epsilon u_t^\epsilon|_{op}^2+\epsilon^2|\omega_t^\epsilon|^2)(|\dot{u}_{t}^{\epsilon}|^2+\epsilon^2 |\dot{\alpha}_{t}^{\epsilon}|^2),$$
	where $|\nabla_t^\epsilon u_t|_{op}$ denotes the operator norm. A crucial observation at this point is the fact that
	$$|\nabla u|_{op}^2+\epsilon^2|\omega|^2=\mz e^\epsilon(u,\nabla)$$
	for entire stationary solutions $(u,\nabla)$ on the plane. Hence, a blow-up analysis analogous to the one used in the previous section gives
	$$ \left|\frac{|\nabla_\tau^\epsilon u_\tau^\epsilon|_{op}^2+\epsilon^2|\omega_\tau^\epsilon|^2}{e^{\epsilon}(u_\tau^\epsilon, \nabla_\tau^\epsilon)}-\mz\right|\le\delta_\epsilon $$
	at all points $(x,\tau)\in (M\times[s,t])\setminus F^\epsilon$, for an exceptional set $F^\epsilon\subseteq M\times[s,t]$ of measure
	$$ \mu^\epsilon(F^\epsilon)\le\delta_\epsilon, $$
	where $\lim_{\epsilon\to0}\delta_\epsilon=0$.
	
	Now, for $p\in(1,2)$, similarly as above we have
	\begin{align*}
		\int_s^t\int_{M} \phi \vert \vec{H}_\tau \vert^p \, d\mu_\tau \, d\tau
		&\le\liminf_{\epsilon \to 0}\int_{M\times[s,t]} \phi\frac{\lvert \divergence(\mathcal{T}_{\tau}^{\epsilon}) \rvert^p}{e^{\epsilon}(u_\tau^\epsilon, \nabla_\tau^\epsilon)^{p-1}} \, \dvol_g\, d\tau \\
		&\le \liminf_{\epsilon \to 0}\int_{M\times[s,t]}4^{p/2}\left(\mz+\delta_\epsilon'\right)\phi e^\epsilon(u_\tau^\epsilon,\nabla_\tau^\epsilon)^{1-p/2}(|\dot{u}_{\tau}^{\epsilon}|^2+\epsilon^2 |\dot{\alpha}_{\tau}^{\epsilon}|^2)^{p/2}\\
		&\quad+\lim_{\epsilon\to0}\int_{F^\epsilon}4^{p/2}\phi e^\epsilon(u_\tau^\epsilon,\nabla_\tau^\epsilon)^{1-p/2}(|\dot{u}_{\tau}^{\epsilon}|^2+\epsilon^2 |\dot{\alpha}_{\tau}^{\epsilon}|^2)^{p/2},
	\end{align*}
	for another sequence $\delta_\epsilon'\to 0$, where on $F^\epsilon$ we invoked the coarser pointwise bound used previously.
	On the other hand, by H\"older's inequality,
	\begin{align*} &\int_{F^\epsilon}e^\epsilon(u_\tau^\epsilon,\nabla_\tau^\epsilon)^{1-p/2}(|\dot{u}_{\tau}^{\epsilon}|^2+\epsilon^2 |\dot{\alpha}_{\tau}^{\epsilon}|^2)^{p/2}\\
	&\le\left(\int_{F^\epsilon}e^\epsilon(u_\tau^\epsilon,\nabla_\tau^\epsilon)\right)^{1-p/2}\left(\int_{F^\epsilon}(|\dot{u}_{\tau}^{\epsilon}|^2+\epsilon^2 |\dot{\alpha}_{\tau}^{\epsilon}|^2)\right)^{p/2}\\&\le\delta_\epsilon^{1-p/2}\Lambda^{p/2}, \end{align*}
	which vanishes in the limit $\epsilon\to 0$. Also, by Young's inequality,
	\begin{align*} &\int_{M\times[s,t]}\phi e^\epsilon(u_\tau^\epsilon,\nabla_\tau^\epsilon)^{1-p/2}(|\dot{u}_{\tau}^{\epsilon}|^2+\epsilon^2 |\dot{\alpha}_{\tau}^{\epsilon}|^2)^{p/2} \\
	&\le \left(1-\frac{p}{2}\right)\int_{M\times[s,t]}\phi e^\epsilon(u_\tau^\epsilon,\nabla_\tau^\epsilon)
	+\frac{p}{2}\int_{M\times[s,t]} \phi (|\dot{u}_{\tau}^{\epsilon}|^2+\epsilon^2 |\dot{\alpha}_{\tau}^{\epsilon}|^2), \end{align*}
	which gives in the limit
	\begin{align*}&\int_s^t\int_{M} \phi \vert \vec{H}_\tau \vert^p \, d\mu_\tau \, d\tau\\
	&\le \frac{4^{p/2}}{2}\left(1-\frac{p}{2}\right)\int_{M\times[s,t]}\phi\,d\mu_\tau\,d\tau
	+\liminf_{\epsilon \to 0}\frac{4^{p/2}}{2}\cdot\frac{p}{2}\int_{M\times[s,t]}\phi(|\dot{u}_{\tau}^{\epsilon}|^2+\epsilon^2 |\dot{\alpha}_{\tau}^{\epsilon}|^2).\end{align*}
	The claim follows once we let $p\to 2$.
\end{proof}

For any smooth vector field $X$ on $M$, recalling that $\sigma=\lim_{\epsilon\to0}\sigma^\epsilon$, we have
\begin{align*}
	\int_s^t\int_M\ang{X,\vec H_\tau}\,d\mu_\tau\,d\tau
	&=\lim_{\epsilon\to0}\int_s^t\int_M\ang{X,\divergence(\mathcal{T}_\tau^\epsilon)}\,\dvol_g\,d\tau\\
	&=-\lim_{\epsilon\to0}\int_s^t\int_M\ang{DX,\mathcal{T}_\tau^\epsilon}\,\dvol_g\,d\tau\\
	&=-\int_s^t\int_M DX \, d\mathcal{T}_\tau\,d\tau,
\end{align*}
where the last integral implicitly contains the usual pairing on $\operatorname{Sym}(TM)$.
Letting $X$ range in a countable dense family, we deduce that $\vec H_t$ is the generalized mean curvature vector of $\mu_t$, for a.e.\ $t$.

In terms of the operators
\begin{equation*}
	\mathcal{B}_\epsilon(u^\epsilon_t, \nabla^\epsilon_t, \phi) := \int_M [- 2 \phi \left( \vert \dot{u}_{t}^{\epsilon} \vert^2 + \epsilon^2 \vert \dot{\alpha}_{t}^{\epsilon} \vert^2 \right) + \divergence(\mathcal{T}_{t}^{\epsilon})(d \phi)],
\end{equation*}
the previous observations give us the upper semicontinuity result
$$\limsup_{\epsilon\to 0}\int_s^t\mathcal{B}(u^{\epsilon}_\tau,\nabla^{\epsilon}_\tau,\phi)\,d\tau\leq \int_s^t\left(\int_M[-\phi|\vec{H}_\tau|^2+\langle d\phi,\vec{H}_\tau\rangle]d\mu_\tau\right)\,d\tau$$
for $\phi\in C^1(M,\R_{\ge0})$ and $0<s<t$ (cf.\ \cite[Section~9]{Ilmanen} in the Allen--Cahn setting). 

Now, recalling \cref{epsBrakke}, for any fixed $\phi\in C^1(M,\R_{\ge0})$, we have 
\begin{equation*}
	\frac{d}{dt} \mu_{t}^{\epsilon}(\phi) = \mathcal{B}_\epsilon(u_{t}^\epsilon, \nabla_{t}^\epsilon, \phi),
\end{equation*}
and integrating this identity in time, we obtain for $s \leq t$, 
\begin{equation*}
	\mu_t^{\epsilon}(\phi) - \mu_{s}^{\epsilon}(\phi) = \int_s^t \mathcal{B}_\epsilon(u_{\tau}^{\epsilon}, \nabla_{\tau}^{\epsilon}, \phi) \,d\tau. 
\end{equation*}
Letting $\epsilon \rightarrow 0$, and appealing to the upper semicontinuity for $\mathcal{B}_{\epsilon}(u_t^{\epsilon},\nabla_t^{\epsilon},\phi)$ described above, we deduce that
\begin{equation}\label{ineq.Brakke}
	\mu_t(\phi) - \mu_s(\phi) \leq \int_s^t \int_M [-\phi\vert \vec{H}_\tau \vert^2 + \langle d \phi, \vec{H}_\tau \rangle] \, d\mu_\tau \,d\tau, 
\end{equation} 
for all $0 < s < t$; actually, recalling the integral bound \cref{br.sum.chain}, it holds also for $s=0<t$, by a trivial limiting argument and the fact that
$\mu_0(\phi)\ge\limsup_{s\to 0}\mu_s(\phi)$.

By virtue of \cref{perpendicularity} below, this conclusion is the same as \cref{brakke.ineq.bis}, namely one of the two forms of Brakke's inequality for families of integral varifolds discussed in \cref{vari.prelim}.
Note that \cref{brakke.sum} holds thanks to \cref{br.sum.chain}.

\begin{remark}\label{perpendicularity}
	For integral varifolds $V$ satisfying the first and second condition of Brakke's inequality assumptions, we have that $\vec{H}$ is orthogonal to $T_xV$ for $|V|$-a.e.\ $x$. This result is due to Brakke: see \cite[Chapter~5]{Brakke}. Alternatively, one can establish perpendicularity of the mean curvature vector reasoning along the same lines as \cite[Section~6]{AmbrosioSoner}. 
\end{remark}

\subsection{Proof of \cref{thmBrakke}} The combined results of Sections \ref{secEpsIneq}--\ref{brakke.sec} give the proof of \cref{thmBrakke}: given a family $(u_t^{\epsilon},\nabla_t^{\epsilon})$ of solutions to the gradient flow system \cref{gf} on $M$, \cref{CorLimMeas} allows us to extract a limiting family of measures $(\mu_t)_{t}$ along a subsequence $\epsilon_j\to 0$. Rectifiability of these measures follows from \cref{rect.coroll}, while the discussion in \cref{sec.int} implies integrality; finally, \cref{ineq.Brakke} gives Brakke's inequality.

\begin{remark}
	Naturally, one can combine the analysis of the preceding sections with arguments from \cite{Ilmanen} to establish other properties of the flow \cref{gf} and the limiting measure $\mu$ analogous to some results of \cite{Ilmanen} in the Allen--Cahn setting. E.g., one can combine a forward lower density bound similar to that of \cite[Section~7]{Ilmanen} with the analysis of \cref{secEpsIneq} above to establish a natural ``equipartition of energy'' result analogous to that of \cite[Section~8]{Ilmanen}: namely, the functions 
	\begin{equation}
		\Xi^\epsilon = \epsilon^2 \vert F_{\nabla^\epsilon} \vert^2 - \frac{1}{4 \epsilon^2}(1 - \vert u^{\epsilon} \vert^2)^2, 
	\end{equation}
	vanish in $L^1_{loc}(M\times [0,\infty))$ as $\epsilon\to 0$. Similar to the proof of \cref{sharper.brakke}, this property also follows from the blow-up analysis from the previous section, combined with the fact that $\Xi^\epsilon=0$ for entire stationary solutions on the plane.
\end{remark}

\section{Enhanced motion in the sense of Ilmanen}\label{EnMo}
In this section, we prove \cref{thmEnhanced} from the introduction, whose statement we recall here for convenience.

\begin{theorem}
	Let $L\to M$ be a Hermitian line bundle over a closed, oriented Riemannian manifold $(M^n,g)$, and let $\Gamma_0$ be an integral $(n - 2)$-cycle Poincaré dual to the Euler class $c_1(L) \in H^2(M; \mathbb{Z})$. Then there exists a family of solutions $(u^{\epsilon}_t,\nabla^{\epsilon}_t)$ of \cref{gf}
	and an integer multiplicity $(n - 1)$-current $\Gamma\in {\bf I}_{loc}^{n-1}( M\times[0,\infty))$ such that $\partial \Gamma = \Gamma_0$, 
	$$\mu_0 =\lim_{\epsilon\to 0}e^{\epsilon}(u^{\epsilon}_0,\nabla^{\epsilon}_0)\,\dvol_g= 2 \pi \vert \Gamma_0 \vert,$$
	and denoting by $\mu_t$ the associated limiting Brakke flow as $\epsilon\to 0$, the pair $(\Gamma, \frac{1}{2\pi} \mu_t)$ defines an enhanced motion in the sense of \cref{defEnh}.
\end{theorem}

\begin{proof} Fix an integral cycle $\Gamma_0$ Poincaré dual to $c_1(L)$. Part (ii) of \cite[Theorem~1.2]{gammaConv} implies the existence of a sequence $(u^\epsilon_0, \nabla^\epsilon_0)$ of smooth sections and connections on $L$ such that 
	\begin{equation*}
		J(u^\epsilon_0, \nabla^\epsilon_0) \rightharpoonup^* 2 \pi \Gamma_0,
	\end{equation*}
	in the sense of currents as $\epsilon \to 0$, where the two-forms $J(u^{\epsilon}_0,\nabla^{\epsilon}_0)$ (as defined in \cref{j.def}) are naturally identified with $(n-2)$-currents via the pairing
	$$\Omega^{n-2}(M)\ni \zeta\mapsto \int_M J(u^{\epsilon}_0,\nabla^{\epsilon}_0)\wedge \zeta,$$
	and we have the weak-$*$ convergence of measures
	\begin{equation*}
		\lim_{\epsilon \rightarrow 0} e^{\epsilon}(u_0^{\epsilon},\nabla_0^{\epsilon})\,\dvol_g = 2 \pi\,d|\Gamma_0|.
	\end{equation*}

	Consider now solutions $(u^\epsilon_t, \nabla^\epsilon_t)_{t \geq 0}$ of the gradient flow \cref{gf} with initial data $(u^\epsilon_0, \nabla^\epsilon_0)$, whose energy measures $\mu_t^{\epsilon}$ converge (subsequentially) as $\epsilon\to 0$ to an integral Brakke flow $(\mu_t)_{t \geq 0}$, by \cref{thmBrakke}. Let $\pi:M\times[0,\infty) \to M$ denote the obvious projection, and along the same subsequence $\epsilon\to 0$, define pairs $(\tilde{u}^{\epsilon},\widetilde{\nabla}^{\epsilon})$ on the pullback bundle $\pi^*L \to  M\times[0,\infty)$ by 
	$$\tilde{u}^{\epsilon}(x,t):=u^{\epsilon}_t(x)$$
	and, for any section $v$,
	$$(\widetilde{\nabla}^{\epsilon})_Xv(x,t):=(\nabla^{\epsilon}_t)_{\pi_*X}v(x,t)+\frac{\partial v}{\partial t}(x,t)\,dt(X).$$
	By direct computation, it is then easy to see that
	$$\widetilde{\nabla}^{\epsilon}\tilde{u}^{\epsilon}(\cdot,t)=\pi^*(\nabla_t^{\epsilon}u^{\epsilon}_t)+\frac{\partial u^{\epsilon}_t}{\partial t}\,dt,$$
	and consequently
	\begin{equation}\label{tilde.u.comp}
		|\widetilde{\nabla}^{\epsilon}\tilde{u}^{\epsilon}|^2(\cdot,t)=|\nabla^{\epsilon}_tu^{\epsilon}_t|^2+\left|\frac{\partial u^{\epsilon}_t}{\partial t}\right|^2.
	\end{equation}
	Writing $\nabla^{\epsilon}_t-\nabla^{\epsilon}_0=:-i\alpha^{\epsilon}_t$ and writing the one-form $\alpha^{\epsilon}_t$ in local coordinates on $M$ as
	$$\alpha^{\epsilon}_t=(\alpha^{\epsilon}_t)_j\,dx^j,$$
	we see that the real two-form $\omega_{\widetilde{\nabla}^{\epsilon}}$ encoding the curvature of $\widetilde{\nabla}^\epsilon$ is given by
	$$\omega_{\widetilde{\nabla}_\epsilon}(\cdot,t)=\pi^*(\omega_{\nabla^{\epsilon}_t})+\sum_{j=1}^n\frac{\partial (\alpha^{\epsilon}_t)_j}{\partial t}\,dt\wedge dx^j.$$
	As a consequence, we have
	\begin{equation}
		|F_{\widetilde{\nabla}^{\epsilon}}|(\cdot,t)^2=|F_{\nabla^{\epsilon}_t}|^2+\left|\frac{\partial \alpha^{\epsilon}_t}{\partial t}\right|^2,
	\end{equation}
	and combining this with \cref{tilde.u.comp}, we see that
	\begin{align*}
		\int_{ M\times[0,T]} e^{\epsilon}(\tilde{u}^{\epsilon},\widetilde{\nabla}^{\epsilon}) &=\int_0^TE_{\epsilon}(u^{\epsilon}_t,\nabla^{\epsilon}_t)\,dt
		+\int_0^T\left(\int_M|\dot{u}^{\epsilon}_t|^2+\epsilon^2|\dot{\alpha}^{\epsilon}_t|^2\right)\,dt.
	\end{align*}
	Controlling the right-hand side via the energy identity \cref{enerID}, we deduce that
	\begin{align*}
		\int_{ M\times[0,T]}e^{\epsilon}(\tilde{u}^{\epsilon},\widetilde{\nabla}^{\epsilon}) &=\int_0^TE_{\epsilon}(u^{\epsilon}_t,\nabla^{\epsilon}_t)\,dt
		+\frac{1}{2}[E_{\epsilon}(u^{\epsilon}_0,\nabla^{\epsilon}_0)-E_{\epsilon}(u^{\epsilon}_T,\nabla^{\epsilon}_T)]\\
		&\leq \left(T+\frac{1}{2}\right)E_{\epsilon}(u^{\epsilon}_0,\nabla^{\epsilon}_0).
	\end{align*}
	In particular, since
	$$\lim_{\epsilon\to 0}E_{\epsilon}(u^{\epsilon}_0,\nabla^{\epsilon}_0)=2\pi \mathbb{M}(\Gamma_0)<\infty,$$
	it follows that 
	\begin{equation}\label{tilde.ener.bd}
		\int_{ M\times[0,T]}e^{\epsilon}(\tilde{u}^{\epsilon},\widetilde{\nabla}^{\epsilon})\leq C (T+1)
	\end{equation}
	for $C$ independent of $\epsilon$. 
	
	For each $N\in \mathbb{N}$, we can then apply \cite[Theorem~1.2(i)]{gammaConv} to deduce the existence of a subsequence $\epsilon_j\to 0$ along which the $(n-1)$-currents associated with the two-forms
	$$\tilde{J}^{\epsilon}:=J(\tilde{u}^{\epsilon},\widetilde{\nabla}^{\epsilon})$$
	converge to $2\pi\Gamma$ for an integer-multiplicity $(n-1)$-current $\Gamma\in {\bf I}^{n-1}( M\times[0,N])$. As written, \cite[Theorem~1.2(i)]{gammaConv} applies to bundles over closed manifolds, but it is easy to see that the same analysis holds over $ M\times[0,N]$; indeed, one could even appeal directly to the closed case by doubling the pairs $(\tilde{u}_{\epsilon},\widetilde{\nabla}_{\epsilon})$ on $ M\times[0,N]$ across the boundary and periodizing, to get a family on the closed manifold $ M\times(\R/2N\mathbb{Z})$.
	
	Moreover, by a simple diagonal sequence argument, we can arrange that this subsequence $\epsilon_j\to 0$ is the same for every $N\in \mathbb{N}$, so that we have the convergence
	$$\tilde{J}^{\epsilon}\rightharpoonup^* 2\pi\Gamma$$
	of $(n-1)$-currents globally on $ M\times[0,\infty)$.
	
	Moreover, recall from \cite{gammaConv} that the two-forms $\tilde{J}^{\epsilon}=J(\tilde{u}^{\epsilon},\widetilde{\nabla}^{\epsilon})$ are closed, and therefore, for any $0\leq a<b<\infty$ and any $(n-2)$-form $\zeta\in \Omega^{n-2}_c( M\times[0,\infty))$, Stokes' theorem gives
	$$(-1)^n\int_{ M\times[a,b]}\tilde{J}^{\epsilon}\wedge d\zeta=\int_{ M\times\{b\}}\tilde{J}^{\epsilon}\wedge \zeta-\int_{ M\times\{a\}}\tilde{J}^{\epsilon}\wedge \zeta.$$
	On the other hand, a simple computation yields
	\begin{equation}\label{tilde.j.comp}\tilde{J}^{\epsilon}(\cdot,t)=\pi^*(J(u^{\epsilon}_t,\nabla^{\epsilon}_t))+\sum_{j=1}^n\left(2\left\langle i\frac{\partial u^{\epsilon}_t}{\partial t},(\nabla_t^{\epsilon})_{\partial_j}u_t^{\epsilon}\right\rangle+(1-|u_t^{\epsilon}|^2)\frac{\partial (\alpha^{\epsilon}_t)_j}{\partial t}\right)\,dt\wedge dx^j,
	\end{equation}
	and since the terms containing a $dt$ component vanish when integrated along a time slice, it follows that
	$$(-1)^n\int_{ M\times[a,b]}\tilde{J}^{\epsilon}\wedge d\zeta=\int_M J(u^{\epsilon}_b,\nabla^{\epsilon}_b)\wedge \zeta-\int_M J(u^{\epsilon}_a,\nabla^{\epsilon}_a)\wedge \zeta.$$
	Taking $a:=0$ and $b$ large enough, we deduce that
	$$(-1)^n\int_{M\times [0,\infty)}\tilde{J}^{\epsilon}\wedge d\zeta=-\int_M J(u^{\epsilon}_0,\nabla^{\epsilon}_0)\wedge \zeta,$$
	and passing to the limit as $\epsilon\to 0$, we deduce that
	$$(-1)^{n+1}\de\Gamma=\Gamma_0\times\{0\},$$
	and of course we can achieve $\de \Gamma=\Gamma_0\times\{0\}$ simply by reversing the orientation of $\Gamma$.
	In particular, $\Gamma_0$ is indeed the slice of $\Gamma$ at time $0$.
	Recalling that the other slices $\Gamma_t\in {\bf I}^{n-2}(M)$ satisfy
	$$\Gamma_t-\Gamma_0=(-1)^n\pi_*[\partial (\Gamma \mres ( M\times[0,t]))],$$
	we deduce from the previous identity (with $a:=0$ and $b:=t$) that
	$$2\pi\Gamma_t=\lim_{\epsilon\to 0}J(u^{\epsilon}_t,\nabla^{\epsilon}_t).$$
	In particular, it follows that, for all $t\in [0,\infty)$ and $\zeta\in\Omega^{n-2}(M),$
	\begin{align*}
		2\pi\int_{\Gamma_t}\zeta&=\lim_{\epsilon\to 0}\int_M J(u^{\epsilon}_t,\nabla_t^{\epsilon})\wedge \zeta \\
		&\leq \liminf_{\epsilon\to 0}\int_M|J(u^{\epsilon}_t,\nabla^{\epsilon}_t)||\zeta|\\
		&\leq \lim_{\epsilon\to 0}\int_M e^{\epsilon}(u_t^\epsilon,\nabla_t^\epsilon)|\zeta|\\
		&=\int |\zeta|\, d\mu_t.
	\end{align*}
	Thus, we see that $|\Gamma_t|\leq \frac{1}{2\pi}\mu_t$, so that condition (v) of \cref{defEnh} is satisfied, and the only condition left to check is (ii), asserting that the measure $B\mapsto \mathbb{M}(\Gamma\mres (M\times B))$ for $B\subseteq \R$ is locally finite and absolutely continuous with respect to the Lebesgue measure. To this end, we proceed as in \cite[pp.~144--146]{BethuelOrlandiSmets}. Given a bounded open set $U\subset[0,\infty)$, it suffices to show that $\mathbb{M}(\Gamma\mres (M\times U))$ is bounded by a multiple of $\sqrt{\mathcal{L}^1(U)}$ when $\mathcal{L}^1(U)\in(0,1)$. Indeed, by \cref{tilde.j.comp}, it is easy to see that
	\begin{align*}
		\int_{M\times U}|\tilde{J}^{\epsilon}|&\leq \int_{M\times U}|\pi^*(J(u_t^{\epsilon},\nabla_t^{\epsilon}))|\\
		&\quad+\int_U\left[\int_M \left(2\left|\frac{\partial u_t^{\epsilon}}{\partial t}\right||\nabla_t^{\epsilon}u_t^{\epsilon}|
		+(1-|u_t^{\epsilon}|^2)\left|\frac{\partial \alpha_t^{\epsilon}}{\partial t}\right|\right)\right]\,dt\\
		&\leq \int_U\left(\int_M e^{\epsilon}(u_t^{\epsilon},\nabla_t^{\epsilon})\right)\,dt\\
		&\quad+\mathcal{L}^1(U)^{-1/2}\int_U\left(\int_M|\nabla_t^{\epsilon}u_t^{\epsilon}|^2+\frac{(1-|u_t|^2)^2}{4\epsilon^2}\right)\,dt\\
		&\quad+\mathcal{L}^1(U)^{1/2}\int_U\left(\int_M\left|\frac{\partial u_t^{\epsilon}}{\partial t}\right|^2+\epsilon^2\left|\frac{\partial \alpha_t^{\epsilon}}{\partial t}\right|^2\right)\,dt\\ 
		&\leq (1+\mathcal{L}^1(U)^{-1/2})\mu^{\epsilon}(M\times U)+\mathcal{L}^1(U)^{1/2}\nu^{\epsilon}(M\times U)\\
		&\leq (\mathcal{L}^1(U)+2\sqrt{\mathcal{L}^1(U)})E_{\epsilon}(u_0^{\epsilon},\nabla_0^{\epsilon}),
	\end{align*}
	where we used Young's inequality, as well as the obvious bound $\mu^{\epsilon}(M\times U)\le\mathcal{L}^1(U)E_{\epsilon}(u_0^{\epsilon},\nabla_0^{\epsilon})$ and \cref{nu.en},
	and passing to the limit $\epsilon\to 0$ gives
	$$\limsup_{\epsilon \to 0}\int_{M\times U}|\tilde{J}^{\epsilon}|\leq 2\pi \mathbb{M}(\Gamma_0)(\mathcal{L}^1(U)+2\sqrt{\mathcal{L}^1(U)}).$$
	Finally, since $\Gamma$ is given as the distributional limit of $\frac{1}{2\pi}\tilde{J}^{\epsilon}$ along some subsequence $\epsilon_j\to 0$, lower semicontinuity of mass under weak-$*$ convergence gives
	$$\mathbb{M}(\Gamma\mres (M\times U))\leq \mathbb{M}(\Gamma_0)(\mathcal{L}^1(U)+2\sqrt{\mathcal{L}^1(U)}),$$
	from which the desired absolute continuity statement follows.
\end{proof}

\nocite{*}
\printbibliography

\end{document}